\documentclass[11pt,reqno]{amsart}

\usepackage[utf8]{inputenc}
\usepackage[paperwidth=8.26in,paperheight=11.69in,
left=1.125in, right=1.125in, bottom=1.5in]{geometry}

\usepackage{amsmath,amsthm,amssymb,enumerate}

\usepackage{thmtools}
\declaretheoremstyle[
headfont=\normalfont\bfseries,
headindent= 0pt,
bodyfont=\em,
spaceabove=8pt,
spacebelow=8pt
]{thm}
\declaretheoremstyle[
headfont=\normalfont\em,
headindent= 0pt,
spaceabove=8pt,
spacebelow=8pt
]{remark}
\declaretheoremstyle[
headfont=\normalfont\bfseries,
headindent= 0pt,
spaceabove=8pt,
spacebelow=8pt
]{example}
\declaretheoremstyle[
headfont=\normalfont\bfseries,
headindent= 0pt,
spaceabove=8pt,
spacebelow=8pt
]{definition}

\declaretheorem[name=Theorem,style=thm,numberwithin=section]{theorem}
\newtheorem*{thm*}{Theorem}
\declaretheorem[name=Proposition,style=thm,sibling=theorem]{proposition}

\declaretheorem[name=Corollary,style=thm,sibling=theorem]{corollary}
\declaretheorem[name=Example,style=definition,sibling=theorem]{example}

\declaretheorem[name=Remark,style=remark]{remark}

\declaretheorem[name=Definition,style=definition,sibling=theorem]{definition}

\usepackage{cleveref}
\crefname{theorem}{Theorem}{Theorems}
\crefname{proposition}{Proposition}{Propositions}
\crefname{lemma}{Lemma}{Lemmas}
\crefname{corollary}{Corollary}{Corollaries}
\crefname{example}{Example}{Examples}
\crefname{definition}{Definition}{Definitions}
\crefname{remark}{Remark}{Remarks}
\crefname{section}{Section}{Sections}
\crefname{enumi}{}{}
\crefname{equation}{}{}

\numberwithin{equation}{section}

\newcommand{\dvol}{d\mathrm{vol}}
\newcommand{\boxop}{\widetilde{\Box}}

\newcommand{\zba}{\bar{z}}

\newcommand{\pol}{\mathbb{D}^n}

\newcommand{\jba}{\bar{j}}
\newcommand{\kba}{\bar{k}}
\newcommand{\lba}{\bar{l}}

\renewcommand{\Re}{\operatorname{Re}}

\usepackage{enumerate}

\newcommand{\RR}{\mathbb{R}}
\newcommand{\CC}{\mathbb{C}}

\DeclareMathOperator{\dom}{dom}

\begin{document}

\title[The $\partial$-operator and real holomorphic vector fields] {The $\partial$-operator and real holomorphic vector fields}

\author{Friedrich Haslinger}
\address{Fakultät für Mathematik, Universität Wien, Oskar-Morgenstern-Platz 1, 1090 Wien, Austria}
\email{friedrich.haslinger@univie.ac.at}
\author{Duong Ngoc Son}
\address{Fakultät für Mathematik, Universität Wien, Oskar-Morgenstern-Platz 1, 1090 Wien, Austria}
\email{son.duong@univie.ac.at}
\date{July 29, 2020}
\thanks{The first-named author was partially supported by the Austrian Science Fund, FWF-Projekt P 28 154-N35. The second-named author was supported by the Austrian Science Fund, FWF-Projekt M 2472-N35.}


\maketitle
\begin{abstract}
	Let $(M,h)$ be a Hermitian manifold and $\psi$ a smooth weight function on~$M$. The $\partial$-complex on weighted Bergman spaces $A^2_{(p,0)}(M,h, e^{-\psi})$ of holomorphic $(p,0)$-forms was recently studied in \cite{hasson} and \cite{haslinger}. It was shown that if $h$ is K\"ahler and a suitable density condition holds, the $\partial$-complex exhibits an interesting holomorphicity/duality property when $(\bar\partial\psi)^{\sharp}$ is holomorphic (i.e., when the real gradient field $\mathrm{grad}_h\psi$ is a real holomorphic vector field.) For general Hermitian metrics this property does not hold without the holomorphicity of the torsion tensor $T_p{}^{rs}$. 
	
	In this paper, we investigate the existence of real-valued weight functions with real holomorphic gradient fields on K\"ahler and conformally K\"ahler manifolds and their relationship to the $\partial$-complex on weighted Bergman spaces. For K\"ahler metrics with multi-radial potential functions on $\CC^n$ we determine all multi-radial weight functions with real holomorphic gradient fields. For conformally K\"ahler metrics on complex space forms we first identify the metrics having holomorphic torsion leading to several interesting examples such as the Hopf manifold $\mathbb{S}^{2n-1} \times \mathbb{S}^1$, and the ``half'' hyperbolic metric on the unit ball. For some of these metrics, we further determine weight functions $\psi$ with real holomorphic gradient fields. They provide a wealth of triples $(M,h,e^{-\psi})$ of Hermitian non-K\"ahler manifolds with weights for which the $\partial$-complex exhibits the aforementioned holomorphicity/duality property.  Among these examples, we study in detail the $\partial$-complex on the unit ball with the half hyperbolic metric and derive a new estimate for the $\partial$-equation.
\end{abstract}
\section{Introduction}
Let $(M,h)$ denote a manifold of complex dimension $n$ with a Hermitian metric $h,$ and let $\psi $ be a smooth real-valued function on $M.$ Consider the Segal-Bargmann spaces 
of $(p,0)$-forms
\begin{equation*}
	A^2_{(p,0)}(M, h, e^{-\psi}) = \left\{ u= \sum_{|J|=p} \, ' u_J\, dz^J : \int_M |u|^2_h e^{-\psi} \dvol_h < \infty, \ u_J  \ {\text{holomorphic}} \right\}.
\end{equation*}
Here $J= (j_1, \dots, j_p)$ are  multiindices of length $p$ and the summation is taken over increasing indices; in holomorphic coordinates, the metric $h$ has the form $h_{j \bar k} dz^j \otimes dz^{\bar k},$ where $[h_{ j \bar k}]$ is a positive definite Hermitian matrix with smooth coefficients; the volume element induced by the metric is denoted by $\dvol_h :=\det (h_{j\bar{l}})\, d\lambda;$  
the metric $h$ induces a metric on tensors of each degree, so for $(1,0)$-forms $u =u_jdz^j$ and $v=v_jdz^j$ one has $\langle u,v \rangle_h = h^{j \bar k} u_jv_{\bar k}$ and $|u|^2_h = \langle u,u \rangle_h,$ where $[h^{j \bar k}]$ is the transpose of the inverse matrix of $[h_{j \bar k}].$  

Under suitable conditions (see \cite{haslinger}, \cite{hasson}) the complex derivative 
$$\partial u := \sum_{|J|=p} \, '  \sum_{j=1}^n \frac{\partial u_J}{\partial z_j} \, dz^j \wedge dz^J $$
is a densely defined, in general unbounded operator  
$$ \partial :A^2_{(p,0)}(M, h, e^{-\psi}) \longrightarrow  A^2_{(p+1,0)}(M, h, e^{-\psi}), \ \  0 \le p \le n-1.$$
In order to determine the adjoint operator
$$\partial^* : A^2_{(p+1,0)}(M, h, e^{-\psi}) \longrightarrow  A^2_{(p,0)}(M, h, e^{-\psi})$$
it is necessary to consider the nonvanishing Christoffel symbols for the Chern connection in local coordinates $z^1,\dots , z^n:$
\begin{equation}
\Gamma^i_{jk} = h^{i\bar{l}} \partial_j h_{k\bar{l}}, \quad \Gamma^{\bar{i}}_{\bar{j}\bar{k}} = \overline{\Gamma^i_{jk}}.
\end{equation}
For a general Hermitian metric, the torsion tensor $T^i_{jk}$ may be nontrivial; it is  defined  by
\begin{equation}
T^i_{jk} = \Gamma^i_{jk} - \Gamma^i_{kj}, \quad T^{\bar{i}}_{\bar{j}\bar{k}} = \overline{T^i_{jk}},
\end{equation}
the torsion $(1,0)$-form is then obtained by taking the trace:
\begin{equation}
\tau = T^{i}_{ji} d z^j.
\end{equation}
We use $h_{j\kba}$ and its inverse $h^{\kba l}$ to lower and raise indices. For example, raising and lowering indices of the torsion, we have
\begin{equation}\label{e:torsion}
	T_q{}^{pr}:
	= T^{\bar{i}}_{\bar{j}\bar{k}} h_{q \bar{i}} h^{p\bar{j}} h^{r\bar{k}}.
\end{equation}
In particular, for a $(0,1)$ form $w = w_{\bar{k}} \, d\bar{z}^k$, raising indices gives the ``musical'' operator $\sharp$ acting on $w$ and to produce an $(1,0)$ vector field $w^{\sharp} : = h^{k\bar{j}} w_{\bar{j}} \, \partial_{k}$. Now, if $(\bar \partial \psi - \bar \tau)^\sharp $ is a holomorphic vector field
the adjoint operator $\partial^*$ on $ \dom (\partial^*) \subset A^2_{(1,0)}(M, h, e^{-\psi})$
can be expressed in the form 
\begin{equation}\label{holvf1}
\partial^* u = \langle u, \partial \psi - \tau \rangle_h,
\end{equation} 
see \cite{hasson} for more details. If, in addition, the metric $h$ is K\"ahlerian one has $\tau = 0$ and thus
\begin{equation}
	\partial^* u = h^{j\bar k} u_j \frac{\partial \psi}{\partial \bar z^k},
\end{equation}
that means that the complex vector field 
\begin{equation}\label{realhol}
X := h^{j\bar k}  \frac{\partial \psi}{\partial \bar z^k} \frac{\partial}{\partial z^j}
\end{equation}
is holomorphic. In this case, the gradient field  $\mathrm{grad}_h \psi$ is a \textit{real holomorphic} vector field in the terminology of \cite{munteanu--wang}. There are important classes of K\"ahler manifolds admitting a function with real holomorphic gradient vector field, for instance the gradient K\"ahler-Ricci solitons, see
\cite{cao} and \cite{munteanu--wang}. The existence of real holomorphic gradient vector fields is also related to Calabi's extremal K\"ahler metric \cite{Cal} and to strong hypercontractivity of the weighted Laplacian \cite{gross1999hypercontractivity}. In \cite{munteanu--wang} it is shown that the real holomorphicity of the gradient vector field of a weight function implies Liouville theorems for weighted holomorphic, or more generally, weighted harmonic functions and mappings on $M.$ We shall see quickly that the holomorphicity of the gradient field of a conformal factor is also related to the holomorphicity of the torsion of the conformal K\"ahler metric.

Here we continue our investigation of the $\partial $-complex
\begin{equation}
	A^2(M, h, e^{-\psi}) 
	\underset{\underset{\partial^* }
		\longleftarrow}{\overset{\partial }
		{\longrightarrow}} A^2_{(1,0)}(M, h,  e^{-\psi}) \underset{\underset{\partial^* }
		\longleftarrow}{\overset{\partial }
		{\longrightarrow}} A^2_{(2,0)}(M, h, e^{-\psi}),
\end{equation}
and the corresponding complex Laplacian
\begin{equation}
	\boxop_p = \partial \partial^{\ast} + \partial^{\ast} \partial :  A^2_{(1,0)}(M, h,  e^{-\psi}) \longrightarrow  A^2_{(1,0)}(M, h,  e^{-\psi}),
\end{equation}
which, under suitable assumptions, will be a densely defined self-adjoint operator, see \cite{hasson} and \cite{haslinger}, where the classical case of the Segal-Bargmann space
with the Euclidean metric is treated.

For $(p,0)$-forms with $p\geqslant 2$, the holomorphicity of $(\bar \partial \psi - \bar \tau)^\sharp$ is not enough for the adjoint $\partial^{\ast}$ to have a simple formula analogous to \cref{holvf1}. In order to describe the formula for $\partial^*$ on $(2,0)$-forms we write 
\begin{equation}
v = \frac{1}{2}\sum_{j,k} v_{jk} dz^j \wedge dz^k = \sum_{j<k} v_{jk} dz^j \wedge dz^k,
\end{equation}
where $v_{jk} = - v_{kj}$. Define an operator $T^{\sharp} \colon \Lambda^{2,0}(M) \to \Lambda^{1,0}(M)$ by
\begin{equation}
		T^{\sharp}(v)
		=
		\frac{1}{2} T_{p}{}^{rs} v_{rs}dz^p.
\end{equation}
where $T_p{}^{rs}$ is given by \cref{e:torsion}. If $u = u_{j} dz^j$, we have
\begin{equation}
	\partial u
	=
	\frac{1}{2}\sum_{j,k} \left(\frac{\partial u_k}{\partial z^j} - \frac{\partial u_{j}}{\partial z^k}\right) dz^j \wedge dz^k.
\end{equation}
Moreover, since $v_{pq} = - v_{qp}$, we find that
\begin{align}
\left\langle \partial u, v\right\rangle_h
=
\sum_{j,k,p,q} \overline{v_{pq}}h^{k\bar{p}} h^{j\bar{q}} \left(\frac{\partial u_k}{\partial z^j}\right).
\end{align}
The formula for $\partial^*$ is then given by 
\begin{equation}\label{e:1243}
	\partial^{\ast} v
	=
	P_{h,\psi}\left(-(\psi_{\bar{j}} - \tau_{\bar{j}}) v_{pq} h^{q\bar{j}} dz^p + T^{\sharp}(v)\right) .
\end{equation}
Here, $P_{h,\psi}$ is the orthogonal projection from $L^2_{(2,0)}(M, h, e^{-\psi})$ onto $A^2_{(2,0)}(M, h, e^{-\psi}),$
see \cite{hasson}. If $h$ is K\"ahler and $(\bar{\partial} \psi)^{\sharp}$ is holomorphic then as in the case of $1$-forms,
\begin{equation}\label{e:2f}
	\partial^{\ast} v
	=
	-\psi_{\bar{j}} v_{pq} h^{q\bar{j}} dz^p .
\end{equation}
That is, the non-local orthogonal projection $P_{h,\psi}$ plays no role and $\partial^{\ast}$ reduces essentially to a ``multiplication'' operator. In the non-K\"ahler case, by inspecting \cref{e:1243}, we find that the relevant condition is the holomorphicity of the torsion tensor; the precise definition is as follows.
\begin{definition}
	Let $h$ be a Hermitian metric on a complex manifold. We say that $h$ has \textit{holomorphic torsion} if
\begin{equation}
	\nabla_{\bar{l}} T_p{}^{rs} =0,
\end{equation}
where $\nabla$ is the Chern connection. 
\end{definition}
Clearly, $h$ has holomorphic torsion if and only if the components of the torsion $T_p{}^{rs}$ (in any holomorphic coordinate frame) are holomorphic. Moreover, it implies that $\bar{\tau}^{\sharp}$ is a holomorphic $(1,0)$ vector field.

Let $D^{\ast}_p$ and $\partial^{\ast}_p$ be the Hilbert space adjoints of $\partial$ in the Lebesgue space $L^2_{(p + 1,0)}(M,h,e^{-\psi})$ and $A^2_{(p + 1,0)}(M,h,e^{-\psi})$, respectively.
In summary, we have the following theorem which generalizes \cite[Theorem 1.1]{hasson}.
\begin{theorem}\label{thm:1.1a}
	Let $(M,h)$ be a complete Hermitian manifold with weight $e^{-\psi}.$  Assume that the torsion $T_{p}{}^{rs}$ of the Chern connection is holomorphic. If $(\bar{\partial} \psi)^{\sharp}$ is holomorphic, then for $\eta \in \dom(D^{\ast}_p)$, $p\geqslant 0$, that is holomorphic in an open set $U\subset M$, $D^{\ast}_p \eta$ is also holomorphic in~$U$. In particular, if $\partial_p$ is densely defined in the Bergman space $A^2_{(p,0)}(M,h,e^{-\psi})$, then 
	\begin{equation}
		D^{\ast}_p \eta = \partial^{\ast}_p\eta
	\end{equation}
	for $\eta \in \dom(\partial^{\ast}_p).$
\end{theorem}
In the following, we give two examples when the theorem applies. The first example shows that in some situations it is necessary to consider non-K\"ahler Hermitian metrics.
\begin{example}[Hopf manifolds]\label{ex1} The simplest examples of Hermitian non-K\"ahler metrics with holomorphic torsion are conformal flat metrics. On $\CC^n$, these metrics are described explicitly in \cref{prop:33}. They are of the form $g_{j\kba} = \phi^{-1} \delta_{jk}$ in the standard coordinates of $\CC^n$, where $\phi$ is given in \cref{e:flata}. For example, in \cref{e:flata}, if we take $c_{j\kba}$ to be $\frac{1}{4} \times$ the identity matrix and $\gamma =0$, then we obtain the following metric on $\CC^n \setminus \{0\}$ with holomorphic torsion:
\begin{equation}\label{e:hopf}
	g_{j\kba} = \frac{4\delta_{jk}}{|z|^2}.
\end{equation}
Let $M:= \mathbb{S}^{n-1} \times \mathbb{S}^1$ be the standard $n$-dimensional Hopf manifold. It is diffeomorphic to $\left(\CC^n \setminus \{0\}\right)/G$, where $G$ is the infinite cyclic group generated by $z\mapsto \frac{1}{2}z$ acting freely and properly discontinuously on $\CC^n\setminus\{0\}$, and has the induced complex structure; see, e.g., \cite{kobayashi} for more details. The Hermitian metric $g_{j\kba}$ in \cref{e:hopf} is invariant under the action of $G$ and descents to a natural locally conformally K\"ahler metric with holomorphic torsion on the standard compact Hopf manifold. It is well-known that for $n\geqslant 2$ the second Betti number $b_2(M) = 0$ and hence $M$ admits \textit{no} K\"ahler metric; see \cite{kobayashi}.
\end{example}
\begin{example} We revisit the following example in \cite{hasson}. Let $M = \mathbb{B}^n$ be the unit ball in $\CC^n$ and let $h_{j\bar{k}} = (1-| z|^2 )^{-1} \delta_{jk}$ be a conformally flat metric. 
By direct computations, we find that the torsion
\begin{equation}
 	T_q{}^{pr}
	= z^p \delta^r_q  - z^r \delta ^p_q
\end{equation}
is nontrivial (unless $n=1$) and holomorphic. Let $\psi = \alpha \log (1 - |z|^2 )$. Then
\begin{equation}
	(\bar{\partial} \psi)^{\sharp} = - \alpha \sum_{j =1}^n z_j \frac{\partial}{\partial z^j}
\end{equation}
is a holomorphic vector field. The triple $(M,h, e^{-\psi})$ satisfies the hypothesis of \cref{thm:1.1a}, except that $h$ is not complete. The $\partial$-complex on the Bergman spaces $A^2_{(p,0)}(M,h, e^{-\psi})$ of holomorphic $(p,0)$-forms exhibits an interesting holomorphicity/duality property similar to that on the Segal-Bargmann space; see \cite{hasson}.
\end{example} 

In this paper, we investigate conformally K\"ahler manifolds with holomorphic torsion and weight functions whose gradients are a real holomorphic vector fields. The first part is devoted to K\"ahler metrics with multi-radial potential functions. It is also shown that in many cases the real holomorphic vector field is of the form 
\begin{equation}
	Z = \sum_{j=1}^n C_jz_j \, \frac{\partial}{\partial z^j},
\end{equation}
where $C_j$ are real constants. In addition, we exploit an example where some constants $C_j$ are zero, which means that the adjoint of $\partial$ ``forgets" some of the variables.

In the second part we consider conformally K\"ahler metrics. Let $(M,h)$ be a K\"ahler manifold and let $g = \phi^{-1}h$ be a conformal metric. We study the condition on $\phi$ such that $g$ has holomorphic torsion. This is the case precisely when $\mathrm{grad}_h \phi$ is a real holomorphic vector field. We determine all conformally K\"ahler metrics having holomorphic torsion on K\"ahler spaces of constant holomorphic sectional curvature. We thus obtain a wealth of examples of Hermitian manifolds with holomorphic torsion.  On some of these examples, we also determine all real-valued functions $\psi$ whose real gradient fields $\mathrm{grad}_g \psi$ are real holomorphic. On such a triple $(M,g,e^{-\psi})$, the $\partial$-complex on the weighted Bergman spaces exhibits an interesting holomorphicity/duality property.  We analyze the $\partial$-complex on the unit ball $\mathbb{B}^n: = \{z \in \CC^n \colon |z|^2 < 1\}$ endowed with the ``half'' hyperbolic metric,
\begin{equation}
	h_{j\kba} = \delta_{jk} + \frac{\zba_j z_k}{1-|z|^2},
\end{equation}
and obtain the following result.
\begin{theorem}[= \cref{prop:51}]\label{thmmain2} Let $h$ be the half hyperbolic metric on the unit ball $\mathbb{B}^n$, $\alpha < 0$, and $\psi(z) = \alpha \log (1-|z|^2)$. Then the complex Laplacian $\widetilde{\Box}_1$ has a bounded inverse $\widetilde{N}_1,$ which is a compact operator on $A^2_{(1,0)}(\mathbb{B}^n, h , e^{-\psi})$ with discrete spectrum. If
	\begin{equation}
	\nu = 
		\begin{cases}
		-\alpha, \quad & \text{if}\ \ n=1, \\
		\min\{1-\alpha, -2\alpha\}, & \text{if}\ \  n =2, \\
		n-\alpha -1, & \text{if}\ \ n\geqslant 3.
		\end{cases}
	\end{equation}
	then
	\begin{equation}
	\left\| \widetilde{N}_1 u \right\| \leqslant \frac{1}{\nu} \, \|u\|,
	\end{equation} 
	for each $u \in A^2_{(1,0)}(\mathbb{B}^n, h , e^{-\psi}).$ In fact, the first positive eigenvalue of $\boxop_1$ is $\lambda_1 = \nu$.
	
	Consequently, if $\eta = \eta_j dz_j \in A^2_{(1,0)}(\mathbb{B}^n, h, e^{-\psi})$ with $\partial \eta =0$, then $f: =\partial^{\ast} \widetilde{N}_1 \eta$ is the canonical solution of $\partial f = \eta, $ this means $\partial f = \eta $ and $f \in (\ker \partial )^\perp$. Moreover, 
	\begin{align}\label{cont61}
	\int_{\mathbb{B}^n} \left|f\right|^2 & (1-|z|^2)^{-\alpha-1} d\lambda \leqslant
	\frac{1}{\nu } \int_{\mathbb{B}^n} \left(\sum_{j=1}^n |\eta_j|^2 - \left|\sum_{j}^{n} \eta_j z_j\right|^2\right)(1-|z|^2)^{-\alpha-1}d\lambda.
	\end{align}
\end{theorem}

We also consider $U(n)$-invariant metrics on $\CC^n$ in a conformal class of a given $U(n)$-invariant K\"ahler metric. It is shown that there exists essentially a 2-parameter family of $U(n)$-invariant conformal metrics with holomorphic and nontrivial torsion. Moreover, with respect to such a metric, there exists essentially a 2-parameter family of weight functions with real holomorphic gradient fields.
\section{K\"ahler metrics with multi-radial potential functions}
We consider K\"ahler metrics on $\CC^n$ with \textit{multi-radial} potential functions 
\begin{equation}
	\chi(z_1,z_2,\dots, z_n)
	=
	\tilde \chi (r_1,r_2, \dots, r_n)
\end{equation}
where $r_j=|z_j|^2, \, j=1, \dots , n$. For these metrics, we can determine explicitly the multi-radial weight functions $\psi$ such that $(\bar\partial\psi)^{\sharp}$ is holomorphic.
\begin{theorem}\label{thm:21}
	Let $\chi(z) = \tilde\chi(|z_1|^2, \dots, |z_n|^2)$ be a multi-radial potential function for a K\"ahler metric in $\mathbb{C}^n$. If $\psi(z) = \widetilde{\psi}(|z_1|^2, \dots, |z_n|^2)$ is a multi-radial weight function such that $(\bar{\partial}\psi)^{\sharp}$ is holomorphic, then
	\begin{equation}\label{e:2.5}
	(\bar{\partial}\psi)^{\sharp}
	=
	\sum_{j=1}^{n} C_j z_j \frac{\partial}{\partial z_j},
	\end{equation}
	where $C_k$'s are real constant and 
	\begin{equation}
		\tilde\psi 
		=
		C_0
		+
		\sum_{j=1}^n C_j r_j \frac{\partial \tilde \chi}{\partial r_j}.
	\end{equation}
\end{theorem}
\begin{proof} By direct computation, we find that 
	\begin{equation}
	h_{j\kba}
	=
	\frac{\partial\tilde \chi}{\partial r_j} \delta_{jk} + \zba_j z_k \left(\frac{\partial^2 \tilde\chi}{\partial r_j \partial r_k}\right).
	\end{equation}
	Observe that $\partial \tilde\chi/\partial r_j$ and $\partial^2 \tilde \chi/\partial r_j \partial r_k$ are real-valued. Observe that $\partial\tilde\chi/\partial r_j > 0$ for all $j$ near the origin.
	
	We claim that  the inverse transpose matrix has the form
	\begin{equation}\label{e:24}
	h^{j\kba} = \left(\frac{\partial\tilde \chi}{\partial r_j}\right)^{-1} \delta_{jk} + V_{jk}z_j\zba_k
	\end{equation}
	for some matrix $V_{jk}$ with real-valued entries.  Indeed, consider the system of equations with unknowns $V_{jk}$
	\begin{equation}
		\left(\left(\frac{\partial\tilde \chi}{\partial r_j}\right)^{-1} \delta_{jk} + V_{jk}z_j\zba_k\right) h_{l\kba} = \delta^j_l
	\end{equation}
	which is equivalent to a system with real coefficients
	\begin{equation}
		\left(\frac{\partial\tilde \chi}{\partial r_j}\right)^{-1}\left(\frac{\partial^2 \tilde\chi}{\partial r_j \partial r_l} \right) +  \left(\frac{\partial\tilde \chi}{\partial r_j}\right)V_{jl} + \sum_{k=1}^n V_{jk} r_k \left(\frac{\partial^2 \tilde\chi}{\partial r_k \partial r_l} \right) = 0.
	\end{equation}
	For fixed $j$, the system of equation for $V_{jk}, k=1,2,\dots , n$ can be written as
	\begin{equation}
		\begin{pmatrix}
			a_j + r_1 b_{11} & r_2b_{21} & \cdots & r_n b_{n1}\\
			r_1 b_{12} & a_j + r_2 b_{22} & \cdots & r_n b_{n2} \\
			\vdots & \vdots & \ddots & \vdots \\
			r_1 b_{1n} & r_2b_{2n} & \cdots & a_j + r_n b_{nn}
		\end{pmatrix} 
		\cdot
		\begin{pmatrix}
			V_{j1} \\
			V_{j2} \\
			\vdots \\
			V_{jn}
		\end{pmatrix} 
		=
		\begin{pmatrix}
		- a_j^{-1} b_{j1}  \\
		- a_j^{-1} b_{j2}  \\
		\vdots \\
		- a_j^{-1} b_{jn} 
		\end{pmatrix},
	\end{equation}
	where
	\begin{equation}
		a_j = \frac{\partial\tilde \chi}{\partial r_j} > 0, \
		\quad
		b_{kl}
		=
		\frac{\partial^2 \tilde\chi}{\partial r_k \partial r_l},
	\end{equation}
	all are real-valued. Clearly, at the origin $r_1 = r_2 = \dots = r_n = 0$, the determinant of the coefficient matrix is $a_j^n >0$. Thus, this system of linear equations is uniquely solvable near the origin and the solution is real. The claim follows. 
	
	On the other hand, since $\psi$ is multi-radial, we have
	\begin{equation}
		\frac{\partial\psi}{\partial \zba_k}
		=
		\frac{\partial\tilde\psi}{\partial r_k} z_k.
	\end{equation}
	This and \cref{e:24} imply that
	\begin{equation}
	(\bar{\partial}\psi)^{\sharp}
	=
	\sum_{j=1}^{n}\left(\left(\frac{\partial\tilde \chi}{\partial r_j}\right)^{-1}\left(\frac{\partial \widetilde{\psi}}{\partial r_j}\right) + \sum_{k=1}^{n}r_k V_{jk}\left(\frac{\partial \widetilde{\psi}}{\partial r_k}\right)\right) z_j \frac{\partial}{\partial z_j}.
	\end{equation}
	Since for each $j$ the expression in the parenthesis is real-valued, it is holomorphic if and only if it is a constant. Thus
	\begin{equation}
		(\bar{\partial}\psi)^{\sharp} = \sum_{j=1}^n C_jz_j \frac{\partial}{\partial z_j},
	\end{equation}
	where $C_1, C_2,\dots , C_n$ are real constants.
	Thus, \cref{e:2.5} holds. Applying the flat ``musical'' operator $\flat$ to both sides, we find that $\tilde\psi$ must satisfy the PDE
	\begin{align}
	z_l\frac{\partial \tilde \psi}{\partial r_l}
		 = 
		\frac{\partial \psi}{\partial \zba_{l}} 
		= 
		\sum_{j=1}^n C_j z_j h_{j\bar{l}} 
		& = 
		z_l C_l \frac{\partial \tilde \chi}{\partial r_l } + z_l\sum_{j=1}^n C_j r_j \frac{\partial^2 \tilde\chi}{\partial r_j \partial r_l} \notag  \\
		& =
		z_l\frac{\partial}{\partial r_l}\left(\sum_{j=1}^n C_j r_j \frac{\partial \tilde \chi}{\partial r_j}\right)
	\end{align}
	whose general solution is
	\begin{equation}
	\tilde\psi 
	=
	C_0
	+
	\sum_{j=1}^n C_j r_j \frac{\partial \tilde \chi}{\partial r_j}.
	\end{equation}
	The proof is complete.
\end{proof}

For example, let 
$\tilde \chi$ have the following form
$$\tilde \chi (r_1,r_2, \dots, r_n)  = F_1(r_1)+F_2(r_2)+ \dots + F_n(r_n) ,$$
where $r_j=|z_j|^2, \, j=1, \dots ,n,$ 
with smooth real valued functions $F_j, \, j=1,\dots,n.$
Then we have a diagonal matrix 
$$h_{j \bar k} = \delta_{jk} (F_j'+r_jF_j'').$$
We have to suppose that all entries satisfy $F_j'+r_jF_j'' >0.$
For the determinant we get 
$$\delta = \prod_{j=1}^n (F_j'+r_jF_j'').$$
For $h^{j\bar k} $ we get
\begin{equation}\label{e:dc}
	(h^{j\bar k} )=  1/\delta
	\begin{pmatrix}  \prod_{j\neq1} (F_j'+r_jF_j'') & 0  & \dots & 0\\
	0 &  \prod_{j\neq 2} (F_j'+r_jF_j'') & \dots & 0\\
	\dots & \dots & \dots & \dots \\
	0 & 0 & \dots &  \prod_{j\neq n} (F_j'+r_jF_j'')
 \end{pmatrix}.
\end{equation}
For this metric, we can always find a weight function $\psi$ such that $(\bar{\partial} \psi)^{\sharp}$ is holomorphic. In fact, we can determine all such multi-radial weight functions $\psi$.
 
\begin{corollary}\label{prop:21}
	Let $h$ be a K\"ahler metric on $\CC^n$ with a potential function
	\begin{equation}
			\chi(z_1,z_2,\dots, z_n)
			=
			\sum_{j=1}^n F_j (|z_j|^2).
	\end{equation}
	If $\psi (z_1,\dots, z_n) =\tilde \psi (|z_1|^2, \dots, |z_n|^2)$ is a multi-radial weight, then $(\bar\partial \psi)^{\sharp}$ is holomorphic if and only if 
	\begin{equation}\label{e:dc1}
		\psi (z_1,\dots, z_n)= C_0+ \sum_{j=1}^n C_j |z_j|^2 F_j'(|z_j|^2) . 
	\end{equation}
	If this is the case, then we obtain the real holomorphic vector field
	\begin{equation}
		h^{j\bar k}  \frac{\partial \psi}{\partial \bar z^k} \frac{\partial}{\partial z^j}= \sum_{j=1}^n C_j z_j \frac{\partial}{\partial z^j}.
	\end{equation}
\end{corollary}
\begin{proof}
	Using \cref{e:dc}, we find that
	\begin{equation}
		h^{ j \bar k} \psi_{\bar k}
		=
		\sum_{j=1}^n z_j\frac{\tilde{\psi}_{r_j}(r_1,r_2,\dots , r_n)}{F'_j + r_j F''_j}.
	\end{equation}
	Then $(\bar\partial \psi)^{\sharp}$ is holomorphic if and only if 
	\begin{equation}
		\frac{\tilde{\psi}_{r_j}(r_1,r_2,\dots , r_n)}{F'_j + r_j F''_j} = C_j
	\end{equation}
	for some real constant $C_j$. This PDE can be solved easily and the solutions are given as in \cref{e:dc1}. The proof is complete.
\end{proof}
\begin{example}
	We consider the polydisk $\pol := \left\{z \in \CC^n \colon |z_j|^2 < 1, j=1,2,\dots, n\right\}$. 
	The Bergman metric on $\pol$ is the K\"ahler metric with potential function
	\begin{equation}
		\chi(z) = \log K(z,z) = -2\sum_{j=1}^n \log (1-|z_j|^2),
	\end{equation}
	which is decoupled and multi-radial. Applying \cref{prop:21} with $F_j(r) = -\log (1-r)$, we see that all multi-radial weight functions $\psi$ with $(\bar\partial \psi)^{\sharp}$ holomorphic are of the form
	\begin{equation}
		\psi = \gamma_0 + \sum_{j=1}^n \frac{ \gamma_j}{1-|z_j|^2}.
	\end{equation}
	Under a suitable condition on $\gamma_j$, the $\partial$-complex on the Bergman spaces $A^2(\pol, h, e^{-\psi})$ is similar to that on the Bergman spaces on the unit ball with complex hyperbolic metric, studied earlier in \cite{hasson}.
\end{example}
Another interesting decoupled multi-radial potential function is given in the form
\begin{equation}
	\tilde \chi (r_1,r_2,\dots, r_n) = \prod_{j=1}^{n} G_j(r_j),
\end{equation}
where $G_j(r)$'s are real-valued function of a real variable. We have 
\begin{align}
	h_{j\jba} 
	& = \partial_j \partial_{\jba} \chi(z_1,z_2,\dots, z_n) \notag \\
	& =
	G_1\cdots G_{j-1}(G_j' + r_j G_j'') G_{j+1} \cdots G_n 
	=
	\chi \frac{G_j' + r_j G_j''}{G_j}
\end{align}
and for $k\ne j$, we have
\begin{equation}
	h_{j\kba} = \partial_j \partial_{\kba} \chi(z_1,z_2,\dots, z_n)
	=
	\chi \frac{G_j'\zba_j G_k' z_k}{G_j G_k}. 
\end{equation}
Thus, the K\"ahler metric is given by a rank-1 perturbation of a diagonal metric. Precisely,
\begin{equation}
	h_{j\kba}
	=
	\chi \left(\frac{G_jG_j' + r_j G_jG_j'' - (G_j')^2 r_j}{G_j^2} \delta_{jk} + \frac{G_j'\zba_j G_k' z_k}{G_j G_k}\right).
\end{equation}
\cref{thm:21} gives the following:
\begin{corollary}
	Let $h$ be a K\"ahler metric on $\CC^n$ with a potential function
	\begin{equation}
		\chi(z_1,\dots , z_n) = \prod_{j=1}^n G_j(|z_j|^2).
	\end{equation}
	Then a multi-radial weight function $\psi(z_1,\dots , z_n) = \tilde{\psi}(|z_1|^2,\dots, |z_n|^2)$ such that $(\bar\partial \psi)^{\sharp}$ is holomorphic if and only if
	\begin{equation}
		\psi(z) = C_0 + \sum_{j=1}^n C_j |z_j|^2 G'_j(|z_j|^2) \prod_{k\ne j} G_k(|z_k|^2),
	\end{equation}
	where $C_0, C_1,\dots, C_n$ are real constants, and $G'_j = \partial G_j/\partial r_j$. In this case, the holomorphic vector field is
		\begin{equation}
	h^{j\bar k}  \frac{\partial \psi}{\partial \bar z^k} \frac{\partial}{\partial z^j}= \sum_{j=1}^n C_j z_j \frac{\partial}{\partial z^j}.
	\end{equation}
\end{corollary}

In the rest of this section, we study in detail an example of K\"ahler metric given by a  multi-radial non-decoupled function, yet the weight function can be chosen so that the adjoint $\partial^{\ast}$-operator  ``forgets'' one variable.
\begin{example}
In the following we consider a non-decoupled example on $\mathbb C^2$ with potential function
\begin{equation}
	\chi (z_1,z_2) = \frac{1}{4}|z_1|^4 + |z_1|^2 |z_2|^2 + |z_1|^2 + |z_2|^2.
\end{equation}
In the standard coordinates of $\CC^2$, the metric is given by the matrix
\begin{equation}
	\left[h_{j\kba}\right]
	=
	\begin{pmatrix}
		|z_1|^2+ |z_2|^2 +1  & \zba_1 z_2 \\
		z_1 \zba_2 & |z_1|^2 +1
	\end{pmatrix},
\end{equation}
with the determinant is 
\begin{equation}
	\delta = \det \left[h_{j\kba}\right] =  |z_1|^4+ 2 |z_1|^2+ |z_2|^2 +1 .
\end{equation}
Therefore,
\begin{equation}
	\left[h^{j\kba}\right]
	=
	\frac{1}{\delta}\begin{pmatrix}
		|z_1|^2 +1 & - z_1 \zba_2 \\
		-\zba_1 z_2 & |z_1|^2 +|z_2|^2+1
	\end{pmatrix}.
\end{equation}

If $\psi(z_1,z_2) = \tilde{\psi}(r_1,r_2)$ is a multi-radial weight with real holomorphic gradient field, then \cref{thm:21} shows that
\begin{equation}
	(\bar\partial\psi)^{\sharp} = C_1 z_1 \frac{\partial}{\partial z_1} + C_2 z_2 \frac{\partial}{\partial z_2},
\end{equation}
where $C_1$ and $C_2$ are two constants, and
\begin{equation}
	\psi(z_1,z_2)
	=
	C_1|z_1|^2 + C_2|z_2|^2 + (C_1+C_2) |z_1z_2|^2 + \frac{1}{2}C_1 |z_1|^4.
\end{equation}

We consider the case $C_1 = 1$ and $C_2=0$ so that 
\begin{equation}
	\psi (z_1,z_2) = \frac{|z_1|^4}{2} + |z_1|^2 |z_2|^2 + |z_1|^2
\end{equation}
and the corresponding Bergman spaces
\begin{equation}
	A_{(0,0)}^2 ( \mathbb C^2, h, e^{-\psi}) =\left\{ f: \mathbb C^2 \longrightarrow \mathbb C \ {\text{entire}} : \int_{\mathbb C^2} |f|^2 e^{-\psi} \delta \, d\lambda < \infty \right\}
\end{equation}
and 
\begin{equation}
	A_{(1,0)}^2 ( \mathbb C^2, h, e^{-\psi}) = \left\{ u=u_1\, dz_1+u_2\, dz_2 ,  u_1,u_2\ {\text{entire}} : \int_{\mathbb C^2} |u|_h^2 e^{-\psi} \delta \, d\lambda < \infty \right\},
\end{equation}
where $|u|^2_h = h^{j \overline k} u_j u_{\overline k}.$ It is easily seen that these spaces are non-trivial.

We claim that $A_{(0,0)}^2 ( \mathbb C^2, h, e^{-\psi}) $ does not contain monomials in $z_2:$
consider the function $f(z_1,z_2) =z_2^m,$ for $m\in \mathbb N.$ Using polar coordinates we get 
\begin{eqnarray} 
	\| f\|^2 &=& 4 \pi^2  \int_0^\infty \int_0^\infty r_2^{2m} (r_1^4 +2r_1^2 + r_2^2 +1)\, e^{-r_1^4/2 - r_1^2r_2^2 -r_1^2} r_1 r_2 \ dr_1 dr_2 \notag \\
	&=& 4 \pi^2  \int_0^\infty  \left ( \int_0^\infty r_2^{2m+3} e^{-r_1^2r_2^2} \, dr_2 \right ) 
	(r_1^5 +2r_1^3 + r_1)e^{-r_1^4/2 - r_1^2} \, dr_1,
\end{eqnarray} 
for the inner integral we substitute $s=r_1^2 r_2^2$ and get
\begin{equation}
	\frac{1}{2 r_1^{2m+4}} \, \int_0^\infty s^{m+1} \, e^{-s} \, ds,
\end{equation}
which shows that integration with respect to $r_1$ is divergent and hence the claim follows.

In a similar way one shows that all functions $z_1^k z_2^\ell,$ for $k \in \mathbb N, k\ge 2$ and $ \ell \in \mathbb Z, 0 \le \ell \le k-2$ belong to
$A_{(0,0)}^2 ( \mathbb C^2, h, e^{-\psi}) .$ They even belong to ${\text{dom}}(\partial ).$ Here one has to take care for the slightly different norm in $A_{(1,0)}^2 ( \mathbb C^2, h, e^{-\psi}) :$
We have to consider the integral
\begin{equation*} 
  \int_0^\infty \int_0^\infty r_1^{2k} r_2^{2\ell} (r_1^4 +2r_1^2 + r_2^2 +1)\, e^{-r_1^4/2 - r_1^2r_2^2 -r_1^2} r_1 r_2 \ dr_1 dr_2;
\end{equation*}
the critical summand is
\begin{equation*} 
  \int_0^\infty \int_0^\infty r_1^{2k} r_2^{2\ell+2} \, e^{-r_1^4/2 - r_1^2r_2^2 -r_1^2} r_1 r_2 \ dr_1 dr_2;
\end{equation*}
integration with respect to $r_2$ gives
\begin{eqnarray*} 
  \int_0^\infty r_1^{2k+1} r_2^{2\ell+2} \, e^{ - r_1^2r_2^2 } r_2 dr_2 &=&
 \frac{1}{2}   \int_0^\infty r_1^{2k-1} r_1^{-2\ell -2} s^{\ell +1} e^{-s}\, ds\\ 
&=&  \frac{1}{2}   \int_0^\infty r_1^{2k-2\ell -3} s^{\ell +1} e^{-s}\, ds,
\end{eqnarray*}
and we observe that $2k-2\ell -3 \ge 0,$ whenever $\ell \le k-2.$

In order to show that the functions  $z_1^k z_2^\ell,$ for $k \in \mathbb N, k\ge 2$ and $ \ell \in \mathbb Z, 0 \le \ell \le k-2$ belong to ${\text{dom}}(\partial ),$ we first have to consider
\begin{equation}
	\partial (z_1^k z_2^\ell) = k z_1^{k-1} z_2^\ell \, dz_1  +  \ell z_1^k z_2^{\ell -1} \, dz_2,
\end{equation}
now we compute 
\begin{eqnarray} 
	| \partial (z_1^k z_2^\ell) |_h^2 &=& \frac{|z_1|^2 +1}{\delta} | k z_1^{k-1} z_2^\ell |^2
	-\frac{z_1 \overline z_2}{\delta} k z_1^{k-1} z_2^\ell \,  \overline{(\ell z_1^k z_2^{\ell-1})} \notag \\ 
	&-&\frac{\overline z_1z_2}{\delta}\overline{( k z_1^{k-1} z_2^\ell)} \,  (\ell z_1^k z_2^{\ell-1})
	+ \frac{|z_1|^2 +|z_2|^2+1}{\delta} |\ell z_1^k z_2^{\ell-1}|^2
\end{eqnarray}
and observe that in the first term the exponent for $r_1$ after integration with respect to $r_2$ is again $2k-2\ell -3$ and in the  last term we have the right exponents for $z_1$ and $z_2,$ namely
$| z_1|^{2k} |z_2|^{2\ell}.$ Hence the functions $z_1^k z_2^\ell,$ for $k \in \mathbb N, k\ge 2$ and $ \ell \in \mathbb Z, 0 \le \ell \le k-2$ belong to ${\text{dom}}(\partial ).$

It is clear that $ \left\{ z_1^k z_2^\ell \  : k \in \mathbb N, k\ge 2, \  \ell \in \mathbb Z, 0 \le \ell \le k-2\right\}$ is an orthogonal system in 
$A_{(0,0)}^2 ( \mathbb C^2, h, e^{-\psi}) .$ 

Let $f\in A_{(0,0)}^2 ( \mathbb C^2, h, e^{-\psi}).$ Then $f$ can be written as its Taylor  series
\begin{equation}
	f(z_1,z_2) = \sum_{\alpha , \beta } c_{\alpha, \beta} z_1^\alpha z_2^\beta,
\end{equation}
which is uniformly convergent on compact subsets of $\mathbb C^2.$ Hence, using polar coordinates we get
\begin{equation}
	\frac{1}{4\pi^2} \int_0^{2\pi} \int_0^{2\pi} f(r_1 e^{i\phi_1}, r_2 e^{i\phi_2}) \, e^{-i\alpha \phi_1} e^{-i\beta \phi_2}  \, d\phi_1 d\phi_2= c_{\alpha, \beta}\,  r_1^\alpha r_2^\beta,
\end{equation}
and by Parseval's formula
\begin{equation}
	\int_0^{2\pi} \int_0^{2\pi} |f(r_1 e^{i\phi_1}, r_2 e^{i\phi_2})|^2 \, d\phi_1 d\phi_2= 4 \pi^2
	\sum_{\alpha , \beta } |c_{\alpha, \beta}|^2  r_1^{2\alpha} r_2^{2\beta}.
\end{equation}
Computing the norm of $f$ in  $A_{(0,0)}^2 ( \mathbb C^2, h, e^{-\psi}),$ we see that
\begin{equation}
	\|f\|^2 = 4 \pi^2 \int_0^\infty \int_0^\infty  \sum_{\alpha , \beta } |c_{\alpha, \beta}|^2  r_1^{2\alpha} r_2^{2\beta} (r_1^4 +2r_1^2 + r_2^2 +1)\, e^{-r_1^4/2 - r_1^2r_2^2 -r_1^2} r_1 r_2 \ dr_1 dr_2,
\end{equation}
and Lebesgue's dominated convergence theorem implies that we can interchange integration and summation, so we have
\begin{equation}
	\|f\|^2= 4 \pi^2 \sum_{\alpha , \beta } \int_0^\infty \int_0^\infty |c_{\alpha, \beta}|^2  r_1^{2\alpha} r_2^{2\beta} (r_1^4 +2r_1^2 + r_2^2 +1)\, e^{-r_1^4/2 - r_1^2r_2^2 -r_1^2} r_1 r_2 \ dr_1 dr_2.
\end{equation}
This implies that  the system $ \{ z_1^k z_2^\ell \  : k \in \mathbb N, k\ge 2, \  \ell \in \mathbb Z, 0 \le \ell \le k-2 \}$ is an orthogonal basis of $A_{(0,0)}^2 ( \mathbb C^2, h, e^{-\psi}),$ as all other functions $z_1^k z_2^\ell $ do not belong to $A_{(0,0)}^2 ( \mathbb C^2, h, e^{-\psi}) .$ In addition we have that the operator $\partial$ is densely defined.

Since $(\bar\partial\psi) ^{\sharp} = z_1\partial/\partial z_1$, we have for $u=u_1\, dz_1+u_2\, dz_2 \in {\text{dom}} (\partial^*)$
\begin{equation}
	\partial^*u = z_1u_1.
\end{equation}
Thus, the adjoint $\partial^{\ast}$ ``forgets'' $z_2$-variable, although the weight and the metric both depend on $z_2$.

Now let $u=u_1dz_1+u_2dz_2 \in  A_{(1,0)}^2 ( \mathbb C^2, h, e^{-\psi}).$ Then
\begin{equation}
	|\partial u|^2_h = \left | \frac{\partial u_2}{\partial z_1}-\frac{\partial u_1}{\partial z_2} \right |^2 \, \frac{1}{\delta},
\end{equation}
therefore 
\begin{equation}
	\partial :  A_{(1,0)}^2 ( \mathbb C^2, h, e^{-\psi}) \longrightarrow  A_{(2,0)}^2 ( \mathbb C^2, h, e^{-\psi})
\end{equation}
is also densely defined.

Let 
\begin{equation}
	v=v_{12} \,dz_1 \wedge dz_2 \in A_{(2,0)}^2 ( \mathbb C^2, h, e^{-\psi}).
\end{equation}
Then, by the same computation as above, we get  
\begin{equation}
	\partial^*v = P_{h,\psi}(-\psi_{\overline j} v_{12} h^{2\overline j})dz_1 +
	P_{h,\psi}(-\psi_{\overline j} v_{21} h^{1\overline j})dz_2 = z_1v_{12}\, dz_2.
\end{equation}
So we obtain for $\tilde \Box = \partial^* \partial + \partial \partial^*$ and $u\in A_{(1,0)}^2 ( \mathbb C^2, h, e^{-\psi}) \cap {\text{dom}}(\tilde \Box) $ that
\begin{equation}
	\tilde \Box u = \left (u_1 + z_1\frac{\partial u_1}{\partial z_1}\right )\, dz_1 
	+z_1 \frac{\partial u_2}{\partial z_1} \, dz_2.
\end{equation}
\begin{proposition}\label{nonradcomp}
The operator
 \begin{equation}
 	\tilde \Box :  A_{(1,0)}^2 ( \mathbb C^2, h, e^{-\psi})\longrightarrow  A_{(1,0)}^2 ( \mathbb C^2, h, e^{-\psi})
\end{equation}
is densely defined and its spectrum consists of point eigenvalues with finite multiplicities. Precisely, for $k=1,2,\dots$, the eigenvalues are $\lambda_k = k+1$, with multiplicity $2k-1$.
\end{proposition}
\begin{proof}
In order to determine the eigenvalues of  $\tilde \Box,$ we consider the basis elements
$z_1^k z_2^\ell,$ for $k \in \mathbb N, k\ge 2$ and $ \ell \in \mathbb Z, 0 \le \ell \le k-2$ and define
\begin{equation}
	v_{k,\ell}^1 = z_1^k z_2^\ell \, dz_1 \ {\text{and}} \  v_{k,\ell}^2 = z_1^k z_2^\ell \, dz_2.
\end{equation}
Then we have 
\begin{equation}
	\tilde \Box v_{k,\ell}^1 = (k+1) v_{k,\ell}^1 \ {\text{and}} \ \tilde \Box v_{k,\ell}^2 = k v_{k,\ell}^2.
\end{equation}
Since $ \ell \in \mathbb Z, 0 \le \ell \le k-2,$ the eigenvalues $k$ and $k+1$ are of finite multiplicity and as the functions $z_1^k z_2^\ell,$ for $k \in \mathbb N, k\ge 2$ and $ \ell \in \mathbb Z, 0 \le \ell \le k-2$ constitute an orthogonal basis in the components of $A_{(1,0)}^2 ( \mathbb C^2, h, e^{-\psi})$ the operator $\tilde \Box $ has  a compact resolvent.
\end{proof}
\end{example}
\section{Conformally K\"ahler metrics}
Let $(M,h)$ be a K\"ahler manifold and let $g = \phi^{-1} h$ be a conformal metric. In this section, we study the question when $g$ has holomorphic torsion. Our motivation comes from \cref{thm:1.1a} which says essentially that if $g$ has holomorphic torsion and if $\psi$ is a weight function such that $(\bar{\partial} \psi)^{\sharp}$ is holomorphic, then the $\partial$-complex on the Bergman spaces $A^2_{(p,0)}(M,h,e^{-\psi})$ exhibits an interesting holomorphicity/duality property, provided that some additional density conditions hold; see also \cite{hasson}. We first consider the case when $(M,h)$ is a complex space form of constant (negative, zero, or positive) curvature. Using a result in \cite{grossqian}, we determine all conformal metrics with holomorphic torsion. We further determine the real-valued function whose gradient with respect to the conformal metrics are real holomorphic. These results provide several interesting examples in which the $\partial$-complex has the aforementioned holomorphicity property.

\begin{proposition}\label{prop:31}
	Let $(M,h)$ be a K\"ahler manifold of dimension $n\geqslant 2$ and let $g = \phi^{-1} h$ be a conformally K\"ahler metric. Let $\tau^g$ be the torsion form of $g$ and $\sharp_g$ the sharp ``musical'' operator associated to $g$. Then the following are equivalent.
	\begin{enumerate}[(i)]
		\item $g$ has holomorphic torsion,
		\item $\left(\overline{\tau}^{g}\right)^{\sharp_g}$ is holomorphic,
		\item $(\bar{\partial} \phi )^{\sharp}$ is holomorphic.
	\end{enumerate}
\end{proposition}	
\begin{proof} ``(i) $\Longrightarrow$ (ii)'' is simple and explained in the introduction. Now let $\hat{\Gamma}_{kl}^j$ and $\hat{T}_{kl}^j$ be the Christoffel symbols and the components of the torsion of $g$ and let $\sigma = -\log \phi$. Then by direct calculation, we have $\hat{\Gamma}_{kl}^j  = \Gamma_{kl}^j + \sigma_k \delta_{l}^j$. Thus,
\begin{align}
	\hat{T}_{kl}^j & = \sigma_k \delta_{l}^j - \sigma_l \delta_{k}^j,\\ \tau^g & = (n-1) \sum_{k=1}^{n} \sigma_k dz_k.
\end{align}
Lowering and raising the indices using $g_{j\kba} = e^{\sigma}h_{j\kba}$ and its inverse
\begin{equation}\label{e:33}
	\hat{T}_p{}^{rs} = 
	\hat{T}^{\jba}{}_{\kba\lba} g^{r\kba} g^{s\lba} g_{p \jba}
	=
	e^{-\sigma}\left(\sigma_{\kba} h^{r \kba} \delta_{p}^s - \sigma_{\bar{l}} h^{s\bar{l}} \delta_p^ r\right)
\end{equation}
and
\begin{equation}
	\left(\overline{\tau^g}\right)^{\sharp_g}
	=
	(n-1) e^{-\sigma} h^{j\kba} \sigma_{\kba} \frac{\partial}{\partial z_j} = (n-1) (\bar{\partial} \phi)^{\sharp}.
\end{equation}
This shows that ``(ii) $\Longleftrightarrow$ (iii)''.
Finally, from \cref{e:33}, $\hat{T}_p{}^{rs}$ is holomorphic if and only if for each $r$, $e^{-\sigma}\sigma_{\kba} h^{r \kba} = \phi_{\kba} h^{r \kba}$ is holomorphic. This shows that (iii) implies (i). The proof is complete.
\end{proof}
Thus, the existence of a conformal metric with holomorphic torsion is equivalent to that of a nonvanishing real-valued solution $\phi$ to the equation $\nabla_j\nabla_k\phi =0$. In many cases considered in this paper, non-constant solutions exist locally or globally on open manifolds. However, we point out that for compact manifolds, the existence of global conformally K\"ahler metrics with nontrivial holomorphic torsion is related to the geometry of the manifolds. In fact, as an application of the ``Bochner technique'' in differential geometry, we have the following
\begin{corollary}
	Let $(M,h)$ be a compact K\"ahler manifold, and let $R_{j \bar k}$ be the Ricci curvature:
	\begin{equation}
		R_{j \bar k} = - \frac{\partial^2}{\partial z_j \partial \bar z_k} \log \det (h_{\ell \bar m}).
	\end{equation}
	 Suppose that $(R_{j \bar k})$ is non-positive. If $g = \phi^{-1} h$ is a conformally K\"ahler metric having holomorphic torsion, then $g$ is homothetic to $h$.
\end{corollary}
For example, there is \textit{no} conformally flat metric with holomorphic torsion on complex flat tori $\CC^n/\Lambda$, $\Lambda$ being a lattice in $\CC^n$, other than the flat metrics.
\begin{proof}
	If $g$ has holomorphic torsion, then by \cref{prop:31}, $(\bar\partial \phi)^{\sharp}$ is holomorphic. By a result of Bochner (see \cite[Theorem~2.4.1]{futaki}), $(\bar\partial \phi)^{\sharp}$ is parallel. In particular, $\bar\partial\partial \phi = 0$ and hence $\phi$ is pluriharmonic. But $M$ is compact and the maximum principle implies that $\phi$ is a constant. 
\end{proof}
\begin{remark} The proof of \cref{prop:31} above is purely local. Thus, we can state a version of ``(i) $\Longleftrightarrow$ (ii)'' for locally conformally K\"ahler manifolds as follows. Recall that if $(M,g)$ is locally conformally K\"ahler, then there exists a closed 1-form $\theta$, the \textit{Lee form}, that satisfies
	\begin{equation}
	d\omega = \theta \wedge \omega, 
	\end{equation}
	where $\omega = i g_{j\kba} dz_j \wedge d\zba_{k}$ is the fundamental $(1,1)$-form in local coordinates (see \cite{moroianu}). Condition (ii) is equivalent to the real holomorphicity of the Lee field $\theta^{\sharp}$. Thus, $g$ has holomorphic torsion if and only if the \textit{Lee vector field} $\theta^{\sharp}$ is holomorphic. This property was studied in, e.g., \cite{moroianu}, which also gives an abundance of conformally K\"ahler metrics on a Hopf manifold (as in \cref{ex1}) with holomorphic Lee field and hence they all have holomorphic torsion.
\end{remark}
\subsection{Conformal flat metrics on $\CC^n$}
\begin{proposition}\label{prop:33} Let $\phi$ be a smooth function such that the set $\{\phi >0 \}$ is a nonempty open set in $\CC^n$. Then, a conformally flat Hermitian metric $g_{j\kba} = \phi^{-1}  \delta_{jk}$ on $\{\phi>0 \}$ has holomorphic torsion if and only if
	\begin{equation}\label{e:flata}
		\phi = \sum_{j,k = 1}^n c_{j\kba} z_j \zba_k +  \Re \sum_{k=1}^n \alpha_k z_k + \gamma,
	\end{equation}
	where $c_{j\kba}$ is a Hermitian matrix, $\alpha_k \in  \CC$ and $\gamma \in \RR$.
\end{proposition}	
\begin{proof} From \cref{prop:31}, the metric $g$ has holomorphic torsion if and only if ${\partial\phi}/{\partial \zba_k}$ is holomorphic for each $k$, or equivalently,
\begin{equation}
	\frac{\partial^2 \phi}{\partial z_j \partial z_j} = 0.
\end{equation}
This PDE has been solved explicitly by Gross and Qian in \cite{grossqian}.
Real-valued solutions to this equation are known to have the form \cref{e:flata}. The proof is complete.
\end{proof}
\begin{example}
	In \cref{e:flata}, if we take $c_{j\kba}$ to be the identity matrix, $\alpha_k = 0$, and $\gamma = 1$, then we obtain on $\CC^n$ a conformally flat Hermitian metric 
	\begin{equation}\label{e:35}
		g_{j\kba} = \frac{\delta_{jk}}{1+|z|^2}.
	\end{equation}
	On the other hand, if we take $c_{j\kba}$ to be minus the identity matrix, $\alpha_k = 0$, and $\gamma = 1$, then we obtain the metric \begin{equation}
		g_{j\kba} = \frac{\delta_{jk}}{1-|z|^2},
	\end{equation}
	which is a conformally flat Hermitian metric on the unit ball $\mathbb{B}^n:=\{|z| <1\}$, cf. \cite{hasson}. Both metrics have holomorphic torsion.

\end{example}
\begin{theorem}\label{thm34}
	Let $M = \mathbb{B}^n$ and let
	\begin{equation}\label{e:37}
	g_{j\kba} = \frac{\delta_{jk}}{1-|z|^2}
	\end{equation}
	be a conformally flat metric on $\mathbb{B}^n$. If $\psi$ is a real-valued function on $\mathbb{B}^n$ such that $(\bar{\partial} \psi)^{\sharp}$ is holomorphic, then
	\begin{equation}
		\psi(z) = A + B \log (1-|z|^2)
	\end{equation}
	for some real constants $A$ and $B$.
\end{theorem}
\begin{proof}
	Let $\psi$ be a weight function on the Hermitian manifold $(\mathbb{B}^n, g)$ such that $(\bar{\partial} \psi)^{\sharp}$ is holomorphic. Since $g^{\kba l} = (1-|z|^2)\delta_{kl}$, we have 
	\begin{equation}
		(\bar{\partial} \psi)^{\sharp} = (1-|z|^2) \sum_{k=1}^n \psi_{\kba} \partial_k.
	\end{equation}
	Thus, the holomorphicity of $(\bar{\partial} \psi)^{\sharp}$ is equivalent to
	\begin{equation}
		f^{(k)}:=(1-|z|^2) \psi_{\kba}
	\end{equation}
	is holomorphic for each $k$. Now, we compute
	\begin{equation}
		\frac{\partial (z_k\psi)}{\partial \zba_{k}} = \frac{z_k f^{(k)}}{1-|z|^2} 
		=
		\frac{\partial}{\partial \zba_k} \left(-f^{(k)}\log (1-|z|^2)\right).
	\end{equation}
	Therefore, 
	\begin{equation}\label{e:311}
		z_k\psi = - f^{(k)}\log (1-|z|^2) + v^{(k)},
	\end{equation}
	where $v^{(k)}$ is holomorphic in $z_k$. Thus, both sides of \cref{e:311} are real-analytic in $z_k$. Expanding in power series at $z_k = 0$ (keeping other variables fixed), we obtain
	\begin{equation}
		f^{(k)}(z) = \sum_{l = 0}^{\infty} A_l z_k^l,
		\quad 
		v^{(k)}(z)  = \sum_{s = 0}^{\infty} C_s z_k^s,
		\quad 
		\psi(z) = \sum_{p,q = 0}^{\infty} c_{pq} z_k^{p} \zba_k^q.
	\end{equation}
Plugging these into equation \cref{e:311} above, we get
	\begin{equation}\label{e:313}
		\sum_{p,q} c_{pq} z_k^{p+1} \zba_k^q
			=
		- \left(\sum_{l = 0}^{\infty} A_l z_k^l	\right) \left(\sum_{m = 0}^{\infty} B_m z_k^m\zba_k^m	\right) + \sum_{s = 0}^{\infty} C_s z_k^s,
	\end{equation}
where 
\begin{equation}
	\log (1-|z|^2) = \sum_{m = 0}^{\infty} B_m z_k^m\zba_k^m,
	\quad 
	B_0 = \log \left( 1- \sum_{j \ne k} |z_j|^2 \right).
\end{equation}
	For each set $(z_j\colon j\ne k)$ fixed, the series involved in equation \cref{e:313} above are uniformly and absolutely convergent in a small disc $\{|z_k| < r \}$. In particular, we can expand the product of infinite sums on the right-hand side and equating the coefficients of monomials $z_k^p \zba_k^q$. Thus, comparing the terms with bi-degree $(p+1, 0)$, we have
	\begin{equation}
		c_{p,0}= - A_{p+1} B_0 + C_{p+1}, \quad p = 0,1,2, \dots.
	\end{equation}
	Comparing terms of bi-degree $(1, q)$ we get
	\begin{equation}
		c_{0,0} = - A_1B_0 + C_1,
		\quad
		c_{0,1} = - A_0 B_1, 
		\quad 
		c_{0, q} = 0\ \text{for}\ q \geqslant 2.
	\end{equation}
	Thus, by the reality of $\psi$, we have 
	\begin{equation}
		c_{p,0} = \overline{c_{0,p}} = 0, \quad \forall\, p \geqslant 2.
	\end{equation}
	Then we find that
	\begin{equation}
		C_p = B_0 A_p \ \text{for}\ p = 0\ \text{and} \ p\geqslant 3.
	\end{equation}
	Hence,
	\begin{equation}
		v^{(k)} = B_0 f^{(k)} + c_{0,0} z_k + c_{1,0} z_k^2.
	\end{equation}
	Plugging this into the original equation \cref{e:311}, we find that
	\begin{equation}\label{e:320}
		-A_0B_1|z_k|^2 + \sum_{p+q \geqslant 2} c_{p,q} z_k^{p+1} \zba_k^q
		=
		- \left(\sum_{l = 0}^{\infty} A_l z_k^l	\right) \left(\sum_{m = 1}^{\infty} B_m z_k^m\zba_k^m	\right).
	\end{equation}
	Equating the terms of bi-degree $(p+1,p)$ we have
	\begin{equation}
		c_{p,p} = -A_1 B_p, \quad p\geqslant 1.
	\end{equation}
	Equating the terms of bi-degree $(p,p)$ we have
	\begin{equation}
		c_{p-1,p} = -A_0 B_p, \quad p \geqslant 1.
	\end{equation}
	Taking the conjugate, we have
	\begin{equation}
		c_{p+1,p} = \overline{c_{p,p+1}} = -\overline{A_0} B_{p+1}.
	\end{equation}
	There's no terms of bi-degree $(p,q)$ if $p< q$ in \cref{e:320}. Thus, $A_3 = A_4 = \dots = 0.$ On the other hand, equating the terms of bi-degree $(p+2,p)$, we find that
	\begin{equation}
		 -A_2 B_p = c_{p+1,p} = -\overline{A_0} B_{p+1}.
	\end{equation}
	This holds for all $p$ if and only if $A_0 = A_2 = 0$ and hence $c_{0,1} = c_{1,0} = 0$. Consequently,
	\begin{equation}
		z_k (\psi - c_{0,0}) = - f^{(k)}(z) \left[\log (1-|z|^2) - B_0\right].
	\end{equation}
	By the reality of $\psi$, $c_{0,0}$, and $\log (1-|z|^2)$, and holomorphicity of $f^{(k)}(z)$ in all variables, we must have
	\begin{equation}
		f^{(k)} (z) = A_1 z_k,
	\end{equation}
	where $A_1$ does not depend on $z_1, z_2, \dots, z_n $. Thus,
	\begin{equation}
		\psi(z) = c_{0,0} + A_1 B_0 + A_1 \log (1-|z|^2)
		=
		C_1 + A_1 \log (1-|z|^2),
	\end{equation}
	where $C_1$ does not depend on $z_k$. To show that $C_1$ is a constant, we assume that $l \ne k$. By the same argument with $k$ is replaced by $l$, we have
	\begin{equation}
		\psi (z) = \tilde{C}_1 + \tilde{A}_1 \log (1 - |z|^2)
	\end{equation}
	for $\tilde{A}_1$ a constant and  $\tilde{C}_1$ does not depend on $z_l$. We have
	\begin{equation}
		\tilde{C}_1 - {C}_1 = (A_1 - \tilde{A}_1)\log (1 - |z|^2).
	\end{equation}
	Applying $\partial^2/\partial z_l \partial z_k$ to both sides, we have
	\begin{align*}
		0 = \frac{\partial^2}{\partial z_l \partial z_k} \left(\tilde{C}_1 - {C}_1\right)
		& =
		\frac{\partial^2}{\partial z_l \partial z_k}
		\left((A_1 - \tilde{A}_1)\log (1 - |z|^2)\right)\\
		& =
		(A_1 - \tilde{A}_1) \zba_k \zba_l (1-|z|^2)^{-2}.
	\end{align*}
	This shows that $A_1 = \tilde{A}_1$ and $C_1 = \tilde{C}_1$. In particular, $C_1$ does not depend on $z_l$, for any~$l$. This completes the proof.
\end{proof}
In Section 5.2 of \cite{hasson}, the authors studied the $\partial$-complex on the weighted Bergman spaces $A^2_{(p,0)}(\mathbb{B}^n, g_{j\kba}, e^{-\psi})$ where $g$ is given in \cref{e:37} above and $\psi (z) = \alpha \log (1-|z|^2)$. \cref{thm34} shows that this choice of the weight function is essentially the only one that makes the $\partial$-complex having the holomorphicity/duality property.

\subsection{Conformal metrics on the complex projective space}
The complex projective space $\mathbb{CP}^n$ is the quotient space
\begin{equation}
	\mathbb{CP}^n = \left(\CC^{n+1} \setminus \{0\}\right)/\sim
\end{equation}
where $\sim$ is the equivalent relation 
\begin{equation}
	(Z_0, Z_1, \dots , Z_n) \sim (Z_0', Z_1' ,\dots, Z_n')
\end{equation}
if and only if $Z_j = \lambda Z_j'$ for some $\lambda\in \CC$. We denote by $[Z_0 \colon Z_1 \colon \cdots \colon Z_n]$ the equivalence class of $(Z_0, Z_1, \dots , Z_n)$ and by $\pi \colon \CC^{n+1} \setminus \{0\} \rightarrow \mathbb{CP}^n$ the canonical projection. Then $\pi$ induces a natural complex manifold structure on $\mathbb{CP}^n$. Moreover, $\mathbb{CP}^n$ is covered by $n+1$ coordinate charts $U_{j}: = \{[Z_0 \colon Z_1 \colon \cdots \colon Z_n] \in \mathbb{CP}^n \colon Z_j \ne 0\}$, $j=0,1,\dots, n$, each of which is biholomorphic to $\CC^{n}$ via the map
\begin{equation}
	\phi_j([Z_0 \colon Z_1 \colon \cdots \colon Z_n]) \to \left(Z_0/Z_j, Z_1/Z_j, \dots, \widehat{Z_j /Z_j}, \dots, Z_n/Z_j\right),
\end{equation}
where the $j^{th}$ coordinate in the right-hand side is removed. The Fubini-Study metric on $\mathbb{CP}^n$ can be described in each coordinates chart $U_j \cong \CC^n$. For example, on $U_0$ the Fubini-Study metric $h_{FS}$ reduces to the K\"ahler metric on $\CC^n$ given by
\begin{equation}
	h_{j\kba} = \partial_j \partial_{\kba} \log (1+ |z|^2), \quad z_k = Z_k/Z_0,\ k = 1,2,\dots, n.
\end{equation}
Then $h_{FS}$ is a K\"ahler metric of constant holomorphic sectional curvature $K = 2$; see \cite{kobayashi}.

\cref{prop:31} and a result of Gross-Qian \cite[\S 3.3]{grossqian} give the following
\begin{proposition} Let $\phi$ be a smooth function such that the set $\{\phi >0 \}$ is a nonempty open set in $\CC^n$. Then, a conformally Fubini-Study Hermitian metric $g_{j\kba} = \phi^{-1} h_{j\kba}$ on $\{\phi>0 \}$ has holomorphic torsion if and only if
	\begin{equation}\label{e:flatb}
	(1+|z|^2)\phi = \sum_{j,k = 1}^n c_{j\kba} z_j \zba_k +  \Re \sum_{k=1}^n \alpha_k z_k + \gamma,
	\end{equation}
	where $c_{j\kba}$ is a Hermitian matrix, $\alpha_k \in  \CC$ and $\gamma \in \RR$.
\end{proposition}
\begin{remark}
	Each function $\phi$ in \cref{e:flatb} give rises to a Hermitian metric conformal to the Fubini-Study metric on a subset $\Omega = \{\phi >0 \}$ of $\CC^n \subset  \mathbb{CP}^n$. Depending on the choice of coefficients, $\Omega$ may be bounded, unbounded, or the whole $\CC^n$. All conformal metrics on the whole $\mathbb{CP}^n$ with holomorphic torsion can be found using a result of Futaki \cite{futaki}. They arise as $\phi^{-1} h_{FS}$ ($h_{FS}$ is the Fubini-Study metric) where $\phi = \phi_0 + C$, where $\phi_0$ is in the first eigenspace of the Laplacian, and $C$ is a real constant, $C > -\min \phi_0$.
\end{remark}
An particular interesting case is when $c_{j\kba} = 0$, $\alpha_k = 0$ and $\gamma = 1$. In this case we have
\begin{equation}\label{e:34}
	g_{j\kba} = \delta_{j\kba} - \frac{\zba_j z_k}{1+|z|^2}
\end{equation}
is a Hermitian non-K\"ahler metric on $\CC^n$ with holomorphic torsion. This metric is analogous to the ``half'' hyperbolic metric on the unit ball discussed in the next section. In the special case $n=1$, this is the same as \cref{e:35} and the metric is the well-known Hamilton's ``cigar'' soliton (a.k.a the Witten's blackhole.)
\begin{theorem} Let $g_{j\kba}$ be as in \cref{e:34}.
	If $\psi$ is real-valued function on $(\CC^n, g_{j\kba})$ such that $(\bar\partial \psi)^{\sharp}$ is holomorphic, then 
	\begin{equation}
	\psi(z) = A + B \log (1+|z|^2)
	\end{equation}
	for $A$ and $B$ are two real constants.
\end{theorem}
The proof of this theorem is similar to that of \cref{thm:36} below. We omit the details.
\subsection{Conformally complex hyperbolic metrics}
Combining \cref{prop:31} and Gross and Qian \cite[Theorem~3.4]{grossqian}, we have the following
\begin{proposition}\label{cor:34}
	Let $\mathbb{B}^n$ be the unit ball in $\CC^n$ and let 
	\begin{equation}
	h_{j\overline k} = (1-|z|^2)^{- 1}\left(\delta_{jk} + \frac{\zba_j z_k}{1-|z|^2}\right)
	\end{equation}
	be the complex hyperbolic metric on $\mathbb{B}^n$. Let $g = \phi^{-1} h$ be a conformal metric on $\mathbb{B}^n$. Then $g$ has holomorphic torsion if and only if 
	\begin{equation}
	(1-|z|^2)\phi = \sum_{j,k} c_{j\kba} z^j \zba^k +  \Re \left(\sum_{k} \alpha_k z^k\right) + \gamma
	\end{equation}
	where $c_{j\kba}$ is a Hermitian matrix, $\alpha_k \in \CC$ and $\gamma \in \mathbb{R}$.
\end{proposition}
\begin{remark}
In \cite{grossqian}, the following example was briefly discussed. For each $\beta \in \RR$, put
\begin{equation}
	h_{j\overline k} = (1-|z|^2)^{\beta - 1}\left(\delta_{jk} + \frac{\zba_j z_k}{1-|z|^2}\right).
\end{equation}
By the Sherman-Morrison formula, we find that the inverse transpose is
\begin{equation}
	h^{k\bar{l}} = (1-|z|^2)^{1-\beta} \left(\delta^{kl} - \zba_ l z_k\right).
\end{equation}
Thus, the torsion tensor takes the following form
\begin{equation}
	T_{jk}^l = \Gamma^l_{jk} -  \Gamma^l_{kj}
	=
	\frac{\beta(\zba_ k \delta^l_j - \zba_ j \delta^l_k)}{1 - |z|^2}
\end{equation}
and $h$ is not K\"ahler, unless $\beta = 0$ or $n=1$. Tracing over the indices $l$ and $k$, we find that
\begin{equation}
	\tau_j = - \frac{\beta (n-1)\zba_ j}{1-|z|^2}.
\end{equation}
Thus, $h$ has holomorphic torsion if and only if $\beta = 0$ (K\"ahler case) or $\beta = 1$. In the latter case, and $h$ is the ``half'' hyperbolic metric which is the only one in this family having holomorphic torsion.
\end{remark}
\begin{theorem}\label{thm:36}
	Let $\psi$ be a function on $\mathbb{B}^n$ with the half hyperbolic metric, then $(\bar{\partial} \psi)^{\sharp}$ is holomorphic if and only if 
	\begin{equation}
		\psi(z) = A + B \log(1-|z|^2),
	\end{equation}
	where $A$ and $B$ are real constants.
\end{theorem}
\begin{proof} Let $Z = Z^{k} \partial_k = (\bar{\partial} \psi)^{\sharp}$. Since
	\begin{equation}
	h^{k\bar{l}} = \delta^{kl} - \zba_ l z_k.
	\end{equation}
	we have 
	\begin{equation}
		Z^k = h^{k\bar{l}} \psi_{\bar{l}}
			= 
			\psi_{\kba} - z_k \sum_{l=1}^{n} \zba_l \psi_{\bar{l}}.
	\end{equation}
	If $Z$ is holomorphic, then
	\begin{equation}
		0 = \partial_{\jba} Z^k
			=
			\psi_{\kba\jba} - z_k \psi_{\jba}- z_k \sum_{l=1}^{n}\zba_l \psi_{\bar{l}\jba} .
	\end{equation}
	Thus,
	\begin{equation}\label{e:339}
		\psi_{\kba\jba} 
		=
		 z_k \psi_{\jba} + z_k \sum_{l=1}^{n}\zba_l \psi_{\bar{l}\jba}
	\end{equation}
	Multiplying both sides with $\zba_{k}$ and summing over $k$, we obtain
	\begin{equation}
		\sum_{l=1}^{n}\zba_{k} \psi_{\kba\jba} 
		=
		|z|^2 \psi_{\jba} + |z|^2\sum_{l=1}^{n}\zba_l \psi_{\bar{l}\jba}.
	\end{equation}
	Therefore,
	\begin{equation} \label{e:341}
		(1-|z|^2)\sum_{l=1}^{n}\zba_{k} \psi_{\kba\jba}  = |z|^2 \psi_{\jba}.
	\end{equation}
	Combining this with \cref{e:339}, we obtain
	\begin{equation}
		\psi_{\kba\jba} = z_k \left(\psi_{\jba} + \frac{|z|^2}{1-|z|^2}\psi_{\jba}\right)
		=
		\frac{z_k \psi_{\jba}}{1-|z|^2}.
	\end{equation}
	Equivalently,
	\begin{equation}
		\frac{\partial}{\partial \zba_{k}} \left[(1-|z|^2) \psi_{\jba}\right] = 0.
	\end{equation}
	Thus, $\psi$ satisfies the conditions in \cref{thm34}. Consequently,
	\begin{equation}
		\psi(z) = A + B \log (1-|z|^2),
	\end{equation}
	where $A$ and $B$ are real constants. The proof is complete.
\end{proof}
\subsection{Conformally $U(n)$-invariant K\"ahler metrics}

In the sequel, we consider $U(n)$-invariant K\"ahler metrics and radial weights. Suppose that $h_{j\bar{k}}$ is a K\"ahlerian metric induced by a radial potential $h(z) = \tilde{h}(|z|^2)$, where $\tilde{h}(r)$ is a real-valued function of a real variable. Precisely, we have
\begin{equation}
	h_{j\overline{k}} = \partial_{j} \partial_{\overline{k}}\, \tilde{h}(|z|^2)
	=
	\tilde{h}'(|z|^2)\, \delta_{jk} + \tilde{h}''(|z|^2)\,  \overline{z}_j z_k.
\end{equation}
Thus, $h_{j\bar{k}}$ is a rank-one perturbation of a multiple of the identity matrix. For $h_{j\overline{k}}$ to be positive definite, we assume that $\tilde{h}'(r) > 0$ and $r\tilde{h}''(r) + \tilde{h}'(r) > 0$. The Sherman-Morrison formula give the formula for the (transposed) inverse 
\begin{equation}\label{a1}
	h^{k \overline{j}}
	=
	\frac{1}{h'} \left(\delta_{jk} - \frac{\tilde{h}'' z_k \overline{z}_j}{\tilde{h}' + r \tilde{h}''}\right), \quad r = |z|^2,
\end{equation}
so that $h_{\overline{l} k} h^{k\overline{j}} = \delta_{\overline{l}}^{\overline{j}}$, the Kronecker symbol.
\begin{proposition}
	Let $g$ be  the conformally $U(n)$-invariant K\"ahler metric 
	\begin{equation}g_{j\bar k} = e^{\tilde \sigma (|z|^2)} \partial_j \partial_{\bar k} \tilde h(|z|^2)\end{equation}
	and $\psi(z) = \tilde \psi(|z|^2)$ is a real-valued radial weight function. Then $(\bar\partial \psi - \bar\tau)^{\sharp} $ is holomorphic if and only if	
	\begin{equation}\label{conf2}
	\tilde \psi (r) = (n-1) \tilde{\sigma}(r) + C_1 \int_0^r e^{\tilde \sigma (s)}(\tilde h'(s) + s \tilde h''(s))\, ds + \tilde C.
	\end{equation}
	where $C$ and $C_1$ are real constants.
\end{proposition}
\begin{proof}
We have
\begin{equation}
	T^{i}_{jk} = \sigma_j \delta^{i}_{k} - \sigma_k \delta^{i}_{j}
	=
	\tilde{\sigma}'\left(\bar{z}_j \delta^i_k - \bar{z}_{k}\delta^i_j\right).
\end{equation}
Then it follows that the torsion $(1,0)$-form of $g_{j\overline k}$ is
\begin{equation}\label{conf1}
	\tau=\tau_k dz^k= (n-1)\tilde{\sigma}'\overline z_k dz^k.
\end{equation}
If $\psi(z) = \tilde{\psi}(r)$, $r = |z|^2$, is a radial weight, then $\partial_{\overline{j}} \psi = \tilde{\psi}'(r) z_j$. For 
\begin{equation}
	(\bar{\partial}\psi -\bar{\tau} )^{\sharp} = g^{j\bar{k}}\left(\psi_{\bar{k}} - \tau_{\bar{k}}\right)\frac{ \partial}{\partial z^j}
\end{equation}
we get 
\begin{equation}
	g^{j\bar{k}}\left(\psi_{\bar{k}} - \tau_{\bar{k}}\right)= \frac{\tilde \psi' - (n-1)\tilde{\sigma}'}
	{e^{\tilde \sigma}(\tilde h' + r \tilde h'')} z_j.
\end{equation}
Therefore $(\bar{\partial}\psi -\bar{\tau} )^{\sharp}$ is holomorphic if and only if
\begin{equation}
	\tilde \psi' = (n-1)\tilde \sigma' + C_1 e^{\tilde \sigma}(\tilde h' + r \tilde h''),
\end{equation}
for some constant $C_1.$ So for another constant $\tilde C$ we have
\begin{equation}\label{conf2}
	\tilde \psi (r) = (n-1) \tilde{\sigma}(r) + C_1 \int_0^r e^{\tilde \sigma (s)}(\tilde h'(s) + s \tilde h''(s))\, ds + \tilde C.
\end{equation}
The proof is complete.
\end{proof}

\begin{example} Considering the unit ball in $\mathbb C^n$ and the hyperbolic metric induced by the potential function $\tilde h (r) = -\log (1-r),$ we get $\tilde h'(r) + r \tilde h''(r) = (1-r)^{-2} $ and 
\begin{equation}\label{conf3}
	\tilde \psi (r) = (n-1) \tilde \sigma (r) + C \int_0^r \frac{e^{\tilde \sigma (s)}}{(1-s)^2}\, ds + C_1.
\end{equation}
Take, for example, $\tilde{\sigma} (r) = \alpha \log (1-r)$, with $\alpha >1$ and 
\begin{equation}
	\tilde{\psi} = \alpha(n-1) \log (1-r) - A (1-r)^{\alpha - 1} + B.
\end{equation}
If $D^{\ast}$ denote the $L^2(M,h,e^{-\psi})$-space adjoint of $\partial$, then $D^{\ast} u $ is holomorphic if $u$ is a holomorphic $(1,0)$-form. However, if $\alpha \ne 0$ and $n \geqslant  3$, then for a holomorphic $(2,0)$-form $v$, $D^{\ast} v$ need not be holomorphic.
\end{example}
\begin{proposition}
	Let $\phi$ be a  radial positive function on $\CC^n$ ($n\geqslant 2$), $\phi(z) = \tilde{\phi}(|z|^2)$. The Hermitian metric $g_{j\bar k} := \phi^{-1} (|z|^2)\partial_j \partial_{\bar k} \tilde h(|z|^2) $
	has holomorphic torsion if and only if
	\begin{equation}
	\tilde{\phi}(r) = A + Br \tilde{h}'(r),
	\end{equation}
	where $A$ and $B$ are two real constants.
\end{proposition}
\begin{proof}
	From \cref{prop:31}, $g$ has holomorphic torsion if and only if $(\bar\partial \phi)^{\sharp}$ is holomorphic. By direct computation,
	\begin{equation}
	h^{j\kba} \phi_{\kba}
	=
	\frac{\tilde{\phi}'(r) z_j}{\tilde{h}'(r) + r \tilde{h}''(r)}, \quad r = |z|^2. 
	\end{equation}
	This is holomorphic for all $l$ if and only if $\tilde{\phi}'(r)/(\tilde{h}'(r) + r \tilde{h}''(r))$ is constant:
	\begin{equation}
	\tilde{\phi}' = B (\tilde{h}' + r \tilde{h}'') = B (r \tilde{h}')'.
	\end{equation}
	Integrating this we complete the proof.
\end{proof}
Hence, the conformally $U(n)$-invariant K\"ahler metric 
\begin{equation}\label{e:un}
	g_{j\bar k} = e^{\tilde \sigma (|z|^2)} \partial_j \partial_{\bar k} \tilde h(|z|^2)
\end{equation}
has holomorphic torsion if and only if 
\begin{equation}
	\tilde \sigma (r) = -\log ( C_2r \tilde h'(r)+C_3),
\end{equation}
where the constant $C_3$ has to be chosen such that $C_2r \tilde h'(r)+C_3 >0.$

This also determines the weight function $\psi:$ we use \eqref{conf2} and get
\begin{equation}\label{a4}
\tilde \psi (r) =- C_4\log(C_2 r\tilde h' (r)+C_3)+ C_5,
\end{equation}
where $C_4= n-1 - (C_1/C_2).$

With this choice of $\tilde \sigma$ and $\tilde \psi$ we get for a $(1,0)$-form $u=u_jdz^j \in {\text{dom}}(\partial^*) $ that
\begin{equation}
	\partial^* u = C_1 \sum_{j=1}^n z_j dz^j,
\end{equation}
and for a $(2,0)$-form $v=v_{pq} dz^p \wedge dz^q \in {\text{dom}}(\partial^*) $ that
\begin{equation}
	\partial^*v = -(C_1- C_2)\sum_{q=1}^n z^q v_{pq}dz^p.
\end{equation}
Finally we have shown the following 

\begin{theorem}\label{determined}
Let $g$ be  the conformally $U(n)$-invariant K\"ahler metric given as in \cref{e:un}
together with an radial real-valued weight function $\psi(z) = \tilde \psi(|z|^2)$.
The vector field $(\bar{\partial}\psi -\bar{\tau} )^{\sharp}$ and the torsion operator $T^{\sharp}$ are holmorphic if and only if
\begin{equation}
	\tilde \sigma (r) = -\log ( C_2r \tilde h'(r)+C_3)
\end{equation}
and
\begin{equation}
	\tilde \psi (r) = - C_4 \log(C_2 r\tilde h' (r)+C_3)+ C_5,
\end{equation}
where $C_4=  n-1-(C_1/C_2)$ and the constant $C_3$ has to be chosen such that $C_2r \tilde h'(r)+C_3 >0.$

In this case we have for the vector field $(\bar{\partial}\psi -\bar{\tau} )^{\sharp}= C_1 \sum_{j=1}^n z^j\partial_j$ 
and for the torsion operator 
$$T^{\sharp}(v) = -C_2\sum_{q=1}^n z^q v_{pq}dz^p.$$
\end{theorem}

\section{The $\partial$-complex on the unit ball with the half hyperbolic metric}
\subsection{The half hyperbolic metric on the unit ball}
Consider the half hyperbolic metric on the unit ball  $\mathbb{B}^n \subset \CC^n$ given in the ``standard'' coordinate by
\begin{equation}
	h_{j\kba} = \delta_{jk} + \frac{\zba_j z_k}{1-|z|^2},
\end{equation}
If $g_{j\kba} = -\partial_j \partial_{\kba} \log (1-|z|^2)$ is the complex hyperbolic metric, then $h_{j\kba} = (1-|z|^2)g_{j\kba}$, i.e., $h$ is conformally K\"ahler.

We have,
\begin{equation}
	\partial_i h_{j\bar{l}} = \frac{\zba_j}{1-|z|^2}\left(\delta_{il} + \frac{\zba_i z_l}{1-|z|^2}\right),
\end{equation}
and therefore,
\begin{equation}
	\Gamma^k_{ij} = h^{k\bar{l}} \partial_{i} h_{j\bar{l}} =  \frac{\zba_j \delta_{ik}}{1-|z|^2}.
\end{equation}
Thus, the curvature of $h$ is
\begin{align}
	R_{i\jba k \lba}
	& =
	- h_{p\lba} \partial_{\jba} \Gamma^p_{ik} \notag \\
	& =
	-\frac{1}{1-|z|^2} \left(\delta_{il}\delta_{jk} + \frac{\delta_{jk} \zba_i z_l}{1-|z|^2} + \frac{\delta_{il} \zba_k z_j }{1-|z|^2} + \frac{\zba_i z_j \zba_k z_l}{(1-|z|^2)^2}\right) \notag \\
	& = 
	-\frac{h_{i\lba} h_{k\jba}}{1-|z|^2}.
\end{align}
Thus, the half hyperbolic metric has negative pointwise constant holomorphic sectional curvature

\begin{equation}
	K(\xi) \bigl|_z = \frac{R_{i\jba k \lba} \xi^i \xi^{\jba}\xi^k \xi^{\lba}}{|\xi|^4} \biggl|_z= -\frac{1}{1-|z|^2}, \quad \text{for}\ \xi = \xi^j \partial_j \in T^{(1,0)}_z(M),
\end{equation}
which is unbounded on $\mathbb{B}^n$. The curvature satisfies additional symmetry
\begin{equation}
	R_{i\jba k \lba} = R_{k\lba i \jba},
\end{equation}
and thus the first two Chern-Ricci curvatures are equal:
\begin{align}
	R^{(1)}_{i\jba}: = h^{k\lba}R_{i\jba k \lba} = -\frac{1}{1-|z|^2} h_{i\jba},\\
	R^{(2)}_{k\lba}: = h^{i\jba}R_{i\jba k \lba} = -\frac{1}{1-|z|^2} h_{k\lba},
\end{align}
and the third Chern-Ricci curvature is
\begin{align}
	R^{(3)}_{k\jba}: = h^{i\lba}R_{i\jba k \lba} = -\frac{n}{1-|z|^2} h_{k\jba}.
\end{align}
The half hyperbolic metric is (weak) Chern-Einstein with two different unbounded and negative Chern scalar curvatures
\begin{equation}
	s: = h^{i\jba}R^{(1)}_{i\jba} = -\frac{n}{1-|z|^2},
	\quad
	\hat{s}:=h^{k\jba} 	R^{(3)}_{k\jba} = -\frac{n^2}{1-|z|^2}.
\end{equation}
Using \cref{e:33}, we find that
\begin{equation}\label{e1}
T_p{}^{rs}
=
z_s \delta^r_p - z_r \delta^s_p
\end{equation}
is holomorphic. Furthermore,
\begin{equation}\label{e2}
\bar{\tau}^{\sharp}
=
-(n-1) \sum_{k=1}^n z_k \frac{\partial}{\partial z_k}
\end{equation}
is also holomorphic.
\subsection{The $\partial$-complex}
\cref{thm:36} suggests that we should choose the weight function
\begin{equation}
\psi(z) =  \alpha \log(1-|z|^2),
\end{equation}
whose gradient is real holomorphic. Since
\begin{equation}
\det [h_{j\kba}] = \frac{1}{1-|z|^2}
\end{equation}
the weighted measure is
\begin{equation}
	e^{-\psi} \dvol_h
	=
	(1-|z|^2)^{-\alpha-1} d\lambda.
\end{equation}
Then the corresponding Bergman space
\begin{equation}
A^2_{(0,0)}(\mathbb{B}^n,h,e^{-\psi})
:=
\left\{f \in \mathcal{O}(\mathbb{B}^n) \colon \|f\|^2:=\int_{\mathbb{B}^n} |f|^2 (1-|z|^2)^{-\alpha - 1} d\lambda < \infty \right\}
\end{equation}
is the ``usual'' Bergman space $A^2_{-\alpha-1}(\mathbb{B}^n)$ in the ball with parameter $-\alpha - 1$, which is of infinite dimension if $\alpha < 0$. We thus assume that $\alpha< 0$ from now on.

For $u = \sum_{k=1}^nu_{k} dz_k$, we have
\begin{equation}
|u|^2_h
:=
u_j u_{\kba} h^{j\kba} = \sum_{k=1}^n |u_k|^2 - \left|\sum_{k=1}^n z_k u_k\right|^2.
\end{equation}
The Bergman space $A^2_{(1,0)}(\mathbb{B}^n, h, e^{-\psi})$ consists of $(1,0)$-forms with holomorphic coefficients $u = \sum_{j=1}^n u_j dz_j$ such that
\begin{equation}
\|u\|^2
:=
\int_{\mathbb{B}^n} \left(\sum_{k=1}^n |u_k|^2 - \left|\sum_{k=1}^n z_k u_k\right|^2\right) (1-|z|^2)^{-\alpha -1} d\lambda < \infty.
\end{equation}

Since the restrictions of polynomials onto $\mathbb{B}^n$ are dense in each Bergman spaces $A^2(\mathbb{B}^n, (1-|z|^2)^{\gamma})$ for $\gamma > -1$, the polynomials as well as $(p,0)$-forms with polynomial coefficients are dense in the respective Bergman spaces. Thus $\partial$-operator is densely defined in $A^2_{(p,0)}(\mathbb{B}^n, h, e^{-\psi})$ for each $0\leqslant p \leqslant n$.

Observe that
\begin{equation}
(\bar\partial\psi)^{\sharp} = -\alpha \sum_{k=1}^n z_k\frac{\partial}{\partial z_k},
\end{equation}
is holomorphic, and by \cref{e2} we have that 
\begin{equation}
(\bar{\partial} \psi - \bar{\tau})^{\sharp}
=
(n-1-\alpha ) \sum_{j=1}^n z_j \partial_j.
\end{equation}
This, together with an integration by parts argument, gives the formula for $\partial^{\ast}$:
\begin{proposition}\label{prop:41}
	Let $u = u_j dz_j \in A^2_{(1,0)}(\mathbb{B}^n,h,e^{-\psi})$. If $\sum_{k=1}^n u_k z_k \in A^2_{(0,0)}(\mathbb{B}^n,h,e^{-\psi})$, then $u\in \dom(\partial^{\ast})$ and
\begin{equation}
\partial^* u = (n-1-\alpha )\sum_{j=1}^n z_j u_j.
\end{equation}
\end{proposition}
\begin{proof} The proof is essentially an integration by parts argument. But the metric $h$ is not complete and thus we need to verify the vanishing of the ``boundary'' term directly. Let $\chi_R$ ($0 < R < 1$) be a smooth family of smooth functions of a real variable such that $\chi_R \equiv 1$ on $(-\infty, R]$, the support of $\chi_R$ is contained in $(-\infty, 1)$, and $|\chi'_R| < 2/(1-R)$. By abuse of notation we write $\chi_R(z) = \chi_R(|z|^2)$, so that $\partial \chi_R/\partial \zba_k = \chi_R'(|z|^2) z_k$.

Let  $v \in A^2_{(0,0)}(\mathbb{B}^n,h,e^{-\psi})$, then by integration by parts,
\begin{align}\label{e:last}
\left(\chi_R u ,\partial v\right)_{L^2(\mathbb{B}^n,h,\psi)}
& =
\int_{\mathbb{B}^n} h^{j\kba} \chi_R u_j \overline{v_k} e^{-\psi} \dvol_h \notag  \\
& = 
\int_{\mathbb{B}^n} \sum_{k=1}^n  \overline{v_k}\left(u_k - \zba_k \sum_{j=1}^n u_j z_j\right) \chi_R(|z|^2) (1-|z|^2)^{-1-\alpha} d\lambda \notag \\
& = 
\int_{\mathbb{B}^n} \overline{v}\sum_{k=1}^n   \frac{\partial}{\partial \zba_{k}}\left(\left(u_k - \zba_k \sum_{j=1}^n u_j z_j\right) \chi_R(|z|^2) (1-|z|^2)^{-1-\alpha}\right) d\lambda \notag \\
& =
(n-\alpha -1)\int _{\mathbb{B}^n} \left(\sum_{k} u_k z_k\right) \chi_R(|z|^2) \overline{v} (1-|z|^2) ^{-\alpha -1} d\lambda \notag \\
& \quad - \int _{\mathbb{B}^n} \overline{v}\left(\sum_{k} u_k z_k\right) \chi_R'(|z|^2)  (1-|z|^2)^{-\alpha} d\lambda .
\end{align}
Since $\chi_R'(|z|^2) = 0$ for $|z|^2 < R$ and $\chi_R'(|z|^2) < 2(1-|z|^2)^{-1}$ for $0\leqslant |z| <1$, we can estimate the last integral as follows:
\begin{align}\label{e:a}
\left|\int\limits _{\mathbb{B}^n} \overline{v}\left(\sum_{k} u_k z_k\right) \chi_R'(|z|^2)  (1-|z|^2)^{-\alpha} d\lambda \right|\leqslant
2\int\limits_{R < |z| < 1} \left|\sum_{k} u_k z_k\right|  |v| (1-|z|^2)^{-\alpha-1} d\lambda.
\end{align}
On the other hand, since both $\sum_{k} u_kz_k$ and $v$ belong to $A^2_{(0,0)}(\mathbb{B}^n,h,e^{-\psi}) = A^2_{-\alpha-1}(\mathbb{B}^n)$, the ``standard'' weighted Bergmann space in the ball with weight $(1-|z|^2)^{-\alpha-1}$, the H\"older inequality implies that
\begin{equation}
	\int_{\mathbb{B}^n} \left|\sum_{k} u_k z_k\right|  |v| (1-|z|^2)^{-\alpha-1} d\lambda
	\leqslant \left\| \sum_{k} u_k z_k \right\|_{A^2_{-\alpha-1}(\mathbb{B}^n)} \cdot \|v\|_{A^2_{-\alpha-1}(\mathbb{B}^n)} < \infty.
\end{equation}
This implies that the right-hand side (and hence both sides) of \cref{e:a} tends to $0$ as $R \to 1^{-}$. Letting $R \to 1^{-}$ in \cref{e:last}, using the denominated Lebesgue convergence theorem, we obtain
\begin{align}
	\left(u ,\partial v\right)_{h,\psi}
	& =
	(n-\alpha -1)\int _{\mathbb{B}^n} \overline{v} \left(\sum_{k} u_k z_k\right) (1-|z|^2) ^{-\alpha -1} d\lambda \notag \\
	& =
	\left((n-\alpha -1)\sum_{k} u_k z_k , v \right)_{h,\psi}.
\end{align}
Consequently, the map $v \mapsto \left(u ,\partial v\right)_{h,\psi}$ is continuous and thus $u \in \dom(\partial^{\ast})$. Moreover,
\begin{equation}
\partial^{\ast} u = (n-\alpha - 1)\sum_{k} u_k z_k.
\end{equation}
The proof is complete.
\end{proof}
For two-form $v_{rs} dz_r \wedge dz_s$, with $v_{rs} = -v_{sr}$, we have by \cref{e1},
\begin{equation}
T^{\sharp} (v) := \frac12 T_p{}^{rs} v_{rs} dz_p
=
\sum_{s=1}^n z_s v_{ps}dz_p.
\end{equation}
Therefore, by \cref{e:1243}, we can verify as in \cref{prop:41} that 
\begin{equation}
\partial^{\ast} v
=
-(n - \alpha -2) \sum_{s=1}^n z_s v_{rs} dz_r.
\end{equation}
For $u = u_j dz_j$, we have
\begin{equation*}
\partial u
=
\frac{1}{2}\sum_{j,k} \left(\frac{\partial u_k}{\partial z_j} - \frac{\partial u_{j}}{\partial z_k}\right) dz_j \wedge dz_k,
\end{equation*}
and thus
\begin{equation}
\partial^{\ast} \partial u
=
(n-\alpha -2) \sum_{k=1}^n\sum_{j=1}^n  \left(\frac{\partial u_k}{\partial z_j} - \frac{\partial u_{j}}{\partial z_k}\right) z_j dz_k. 
\end{equation}
On the other hand, 
\begin{equation}
\partial\partial^{\ast}u
=
(n-\alpha -1) \sum_{k=1}^n \left( u_k + \sum_{j=1}^n z_j \frac{\partial u_j}{\partial z_k}\right)\, dz_k.
\end{equation}
Consequently,
\begin{equation}
\widetilde{\Box}_1 u
=
(n-\alpha -1) u + \sum_{k=1}^n \sum_{j=1}^n \left((n-\alpha - 2) \frac{\partial u_k}{\partial z_j} +  \frac{\partial u_j}{\partial z_k}\right)z_jdz_k.
\end{equation}
Unlike the cases of Segal-Bargmann space \cite{haslinger} and weighted Bergman space with hyperbolic metric \cite{hasson}, this is not a diagonal operator. Nevertheless we can apply the methods from Theorem 5.4 of \cite{hasson} to get the following
\begin{theorem}\label{prop:51} Let $h$ be the half hyperbolic metric on the unit ball $\mathbb{B}^n$, $\alpha < 0$, and $\psi(z) = \alpha \log (1-|z|^2)$. Then the complex Laplacian $\widetilde{\Box}_1$ has a bounded inverse $\widetilde{N}_1,$ which is a compact operator on $A^2_{(1,0)}(\mathbb{B}^n, h , e^{-\psi})$ with discrete spectrum.	In addition, if
	\begin{equation}
	\nu = 
		\begin{cases}
		-\alpha, \quad & \text{if}\ \ n=1, \\
		\min\{1-\alpha, -2\alpha\}, & \text{if}\ \  n =2, \\
		n-\alpha -1, & \text{if}\ \ n\geqslant 3.
	\end{cases}
	\end{equation}
	then
	\begin{equation}
	\left\| \widetilde{N}_1 u \right\| \leqslant \frac{1}{\nu} \, \|u\|,
	\end{equation} 
	for each $u \in A^2_{(1,0)}(\mathbb{B}^n, h , e^{-\psi}).$
	
	Consequently, if $\eta = \eta_j dz_j \in A^2_{(1,0)}(\mathbb{B}^n, h, e^{-\psi})$ with $\partial \eta =0$, then $f: =\partial^{\ast} \widetilde{N}_1 \eta$ is the canonical solution of $\partial f = \eta, $ this means $\partial f = \eta $ and $f \in (\ker \partial )^\perp$. Moreover, 
	\begin{align}\label{cont61}
	\int_{\mathbb{B}^n} \left|f\right|^2 & (1-|z|^2)^{-\alpha-1} d\lambda \leqslant
	\frac{1}{\nu } \int_{\mathbb{B}^n} \left(\sum_{j=1}^n |\eta_j|^2 - \left|\sum_{j}^{n} \eta_j z_j\right|^2\right)(1-|z|^2)^{-\alpha-1}d\lambda.
	\end{align}
\end{theorem}
\begin{remark}
	If $n=1$ or $n\geqslant 3$, then the first positive eigenvalue of $\boxop_1$ is $\lambda_1=n-1-\alpha$ with the multiplicity $n$. If $n=2$, there are three subcases: If $-1 < \alpha < 0$, then $\lambda_1 = -2\alpha$ is a simple eigenvalue and the corresponding eigenspace $E_1$ is spanned by $z_1dz_2 - z_2 dz_1$; if $\alpha = -1$, then $\lambda_1 = 2$ with multiplicity 3 and $E_1$ is spanned by $dz_1, dz_2$, and $ z_1dz_2 - z_2 dz_1$; if $\alpha < -1$, then $\lambda_1 = 1-\alpha$ with multiplicity 2 and $E_1$ is spanned by $dz_1$ and $dz_2$.
\end{remark}
\begin{proof}
The subspaces
\begin{equation}
A^2_{(1,0)}(m) := \mathrm{span}\, \left\{c_J z^J dz_l \colon, |J| = m, l = 1,2,\dots, n \right\}, \quad m = 0, 1,2,\dots 
\end{equation}
are invariant under the action of $\boxop_1$. Using a standard result in spectral theory (see Lemma~5.1 of \cite{hasson} or \cite{davies1995spectral}), we can study the spectrum of $\boxop_1$ by study the spectra of its restrictions onto finite dimensional subspaces $A^2_{(1,0)}(m)$.  If $n=1$, then each subspace is one-dimensional. Moreover, write $z_1 = z$, we have
\begin{equation}
\boxop_1(z^k dz) = -(k+1) \alpha z^k dz.
\end{equation}
We find that, when $n=1$, $\boxop_1$ has simple eigenvalues $-\alpha, -2\alpha, \dots \to +\infty$ since $\alpha < 0$.

Consider the case $n\geqslant 2$. When $m = 0$, $A^2_{(1,0)}(0)$ is spanned by $dz_1, dz_2, \dots , dz_n$ and $\boxop_1 (dz_k) = (n-\alpha -1)\, dz_k$ and hence $n-\alpha -1$ is an eigenvalue for $\boxop_1$. When $m = 1$, $A^2_{(1,0)}(1)$ has dimension $n^2$ and is spanned by $z_j dz_k$, $j,k = 1, \dots n$. For example, if $n=2$ then the matrix representation of $\boxop_1$ in the basis $e_1:= z_1 dz_1, e_2:=z_1 dz_2, e_3 := z_2 dz_1$, and $e_4:=z_2 dz_2$ is the following constant column-sum matrix
\begin{equation}
\begin{pmatrix}
2-2 \alpha  & 0 & 0 & 0 \\
0 & 1-2\alpha & 1 & 0 \\
0 & 1 & 1-2\alpha & 0 \\
0 & 0 & 0 & 2-2\alpha  \\
\end{pmatrix}
\end{equation}
whose eigenvalues are $-2\alpha$ and $2(1-\alpha)$, the latter has  multiplicity~3, and the matrix is diagonalizable. Observe that $-2\alpha$ is an eigenvalue for all $n\geqslant 2$.

Consider the case $m=2$ and $n=2$, $A^2_{(1,0)}(2)$ has a basis of 6 vectors: $e_1 = z_1^2 dz_1$, $e_2 = z_1^2 dz_2$, $e_3 = z_1z_2 dz_1$, $e_4 = z_1z_2 dz_2$, $e_5 = z_2^2 dz_1$, and $z_6 = z_2^2 dz_2$. The matrix representation of $\boxop_1$ in this basis is
\begin{equation}
\begin{pmatrix}
3-3 \alpha  & 0 & 0 & 0 & 0 & 0 \\
0 & 1-3\alpha & 1 & 0 & 0 & 0 \\
0 & 2 & 2-3\alpha & 0 & 0 & 0 \\
0 & 0 & 0 & 2-3\alpha & 2 & 0 \\
0 & 0 & 0 & 1 & 1-3\alpha  & 0 \\
0 & 0 & 0 & 0 & 0 & 3 - 3 \alpha
\end{pmatrix}
\end{equation}
The eigenvalues of this matrix are $3(1-\alpha)$ (multiplicity 4) and $-3\alpha$ (multiplicity 2).

Let $\Lambda = (\lambda_1,\lambda_2,\dots , \lambda_n)$ be a multi-index and let $|\Lambda| = \lambda_1+ \lambda_2 + \cdots + \lambda_n$. If $k\ne l$, we define the multi-index
\begin{equation}
	\Lambda_{j,l} = (\lambda_1, \dots, \lambda_{l-1}, \lambda_l+1, \lambda_{l+1}, \dots , \lambda_{j-1}, \lambda_j-1, \lambda_{j+1}, \dots, \lambda_n),
\end{equation}
when $j > l$ and similarly for $l<j$. That is, the operation $\Lambda \mapsto \Lambda_{j,l}$ adds 1 to $l^{th}$-index and subtracts $1$ from $j^{th}$-index. Clearly, $|\Lambda_{j,l}| = |\Lambda|$.

If $u = z^{\Lambda} dz_l$ where $\Lambda = (\lambda_1,\lambda_2,\dots , \lambda_n)$ is a multi-index, then
\begin{align}
	\boxop_1 u 
	& = \left( (|\Lambda|+1)(n-\alpha -1) - |\Lambda| + \lambda_l\right) z^{\Lambda} dz_l + \sum_{j\ne l}\lambda_j z^{\Lambda_{j,l}} dz_j.
\end{align}

Suppose that $e_{\gamma} = z^{\Lambda^{\gamma}} dz_{l_{\gamma}} \ , \Lambda^\gamma =( \lambda_1^\gamma, \dots ,\lambda_n^\gamma ),  \ |\Lambda_{\gamma}| = m,  \ \gamma = 1,2,\dots, N$, be a basis for the space $A^2_{(1,0)}(m)$. Write
\begin{equation}
	\boxop_1(e_{\beta}) = \sum_{\gamma} a_{\gamma \beta} e_{\gamma}.
\end{equation}
The matrix representation for $\boxop_1$ on $A^2_{(1,0)}(m))$ is a constant sum column matrix; the sum of the entries of each column is
\begin{equation}
\sum_{\beta}^N a_{\gamma\beta}
=
(m+1)(n-\alpha-1), \quad N = n\binom{n+m-1}{n-1},
\end{equation}
while the diagonal entries are of the form
\begin{align}\label{gen}
(m+1)(n-\alpha-1) + \lambda_{l} - m.
\end{align}
Take $\gamma \ne \beta$.
Clearly, if $l_{\gamma} = l_{\beta}$ then $a_{\gamma\beta} = 0$. If $l_\gamma \ne l_{\beta}$ and if $\Lambda^{\beta}_{l_{\gamma},l_{\beta}} \ne \Lambda^{\gamma}$, then $a_{\gamma\beta} = 0$.
Finally, if $l_\gamma \ne l_{\beta}$ and $\Lambda^\beta _{l_{\gamma},l_{\beta}} = \Lambda^{\gamma}$, then 
\begin{equation}
	a_{\gamma\beta} = \lambda^{\beta}_{l_{\gamma}} = \lambda^{\gamma}_{l_{\gamma}} + 1.
\end{equation}
Thus, we have for each fixed $\gamma$,
\begin{equation}
	\sum_{\beta} a_{\gamma\beta}\ = \sum_{\beta,l_{\gamma} \ne l_{\beta},\Lambda^\beta _{l_{\gamma},l_{\beta}} = \Lambda^{\gamma}} (\lambda^{\gamma}_{l_{\gamma}} + 1)
	=
	q_{\gamma} (\lambda^{\gamma}_{l_{\gamma}} + 1).
\end{equation}
where $q_{\gamma}$ equals the number of nonzero index in the multi-index $\Lambda^{\gamma}$ other than $\lambda_{l_{\gamma}}$; in particular, $q_{\gamma} \leqslant n-1$. We first consider the case $\lambda_{l_{\gamma}}^{\gamma} \leqslant m-2$. Then and \cref{gen} show that 
\begin{align}
	\delta_{\gamma}:=a_{\gamma\gamma} - \sum_{\beta \ne \gamma}a_{\gamma \beta}
	& \geqslant 
	((m+1)(n-\alpha-1) + \lambda^{\gamma}_{l_{\gamma}} - m) - (n-1)(\lambda^{\gamma}_{l_{\gamma}} + 1) \notag \\
	& =
	-\alpha (m+1) + (n-2) (m-\lambda^{\gamma}_{l_{\gamma}}) \notag \\
	& \geqslant -\alpha (m+1) + 2(n-2).\notag  \\
	& \geqslant 2(n-\alpha-2).
\end{align}
If $\lambda_{l_{\gamma}}^{\gamma} = m-1$, then $q_{\gamma} = 1$ and in this case $\delta_{\gamma} = (m+1)(n-\alpha-2)$. If $\lambda_{l_{\gamma}}^{\gamma} = m$, then $q_{\gamma} = 0$ and $\delta_{\gamma} = (m+1)(n-\alpha - 1)$. Thus, in any case
\begin{equation}
	\delta_{\gamma} \geqslant 2(n-\alpha-2).
\end{equation}
By theorem of Ger{\v{s}}gorin \cite{gervsgorin1931abgrenzung}, the eigenvalues of $[a_{\alpha\beta}]$ must be in the union of the circles centered at $a_{\gamma\gamma}$ with radius $R_{\gamma} = a_{\gamma\gamma}-\delta_\gamma$, $\gamma = 1,2,\dots, N$. Consequently, the eigenvalues must be larger than $2(n-\alpha-2)$. Moreover, for $m\geqslant 2$, these eigenvalues of $\boxop_1$ on $A^2_{1,0}(m)$ are larger than $-\alpha(m+1) \to \infty$. This shows that the bounded inverse $\widetilde{N}_1$ is a bounded and compact operator.

When $n = 2$, $2(n-\alpha-2) = -2\alpha$ is an eigenvalue and the corresponding eigenspace in $A^2_{(1,0)}(1)$ is spanned by $z_1dz_2 - z_2dz_1$. Thus the first positive eigenvalue in this case is 
\begin{equation}
	\lambda_1
	=
	\min \{ 1-\alpha, -2\alpha\}.
\end{equation}

When $n\geqslant 3$, we always have $2(n-\alpha-2) > n-\alpha -1$ since $\alpha <0$ and thus $\lambda_1 = n-\alpha -1$.
The proof is complete.
\end{proof}

\end{document}